\documentclass[11pt,leqno]{article}

\usepackage{amsmath,amsthm}

\usepackage{amssymb,latexsym}

\usepackage{enumerate}

\newcommand{\FF}{{\cal F}}

\newcommand{\LL}{{\cal L}}

\newcommand{\NN}{{\cal N}}

\newcommand{\TT}{{\cal T}}

\newcommand{\WW}{{\cal W}}

\newcommand{\BN}{{\mathbb{N}}}
\newcommand{\BR}{{\mathbb{R}}}
\newcommand{\BX}{{\mathbb{X}}}

\newcommand{\fch}{\mathbf{1}}

\newtheorem{theorem}{\bf Theorem}[section]
\newtheorem{proposition}[theorem]{\bf Proposition}
\newtheorem{lemma}[theorem]{\bf Lemma}
\newtheorem{corollary}[theorem]{\bf Corollary}

\theoremstyle{definition}
\newtheorem*{definition}{Definition}
\newtheorem{example}[theorem]{\bf Example}

\newtheorem{remark}[theorem]{Remark}

\numberwithin{equation}{section}

\begin{document}

\title {The valuation of American options in a multidimensional exponential
L\'evy model}
\author {Tomasz Klimsiak and Andrzej Rozkosz}
\date{}
\maketitle

\begin{abstract}
We consider the problem of valuation of American options written
on dividend-paying assets whose price dynamics follow a
multidimensional exponential L\'evy model. We carefully examine
the relation between the option prices, related partial
integro-differential variational inequalities and reflected
backward stochastic differential equations. In particular, we
prove regularity results for the value function and obtain the
early exercise premium formula for a broad class of payoff
functions.
\end{abstract}

\footnotetext{This work was supported by Polish National Science Centre Grant No.
2012/07/B/ST1/03508.}

\footnotetext{T. Klimsiak: Institute of Mathematics, Polish
Academy of Sciences, \'Sniadeckich 8, 00-956 Warszawa, Poland, and
Faculty of Mathematics and Computer Science, Nicolaus Copernicus
University, Chopina 12/18, 87-100 Toru\'n, Poland; e-mail:
tomas@mat.umk.pl.}

\footnotetext{A. Rozkosz: Faculty of Mathematics and Computer
Science, Nicolaus Copernicus University, Chopina 12/18, 87-100
Toru\'n, Poland; e-mail: rozkosz@mat.umk.pl.}

\section{Introduction}
\label{sec1}

In this paper, we consider the problem of valuation of American
options in  a market model consisting of $d\ge1$ assets whose
prices $X^{s,x,1},\dots,X^{s,x,d}$ on the time interval $[s,T]$
under some risk-neutral probability measure $P$ are represented by
\begin{equation}
\label{eq1.1}
X^{s,x,i}_t=x_ie^{(r-\delta_i)(t-s)+\xi^i_t-\xi^i_s},\quad
t\in[s,T],\quad i=1,\dots,d.
\end{equation}
Here,  $x_i>0$, $i=1,\dots,d$, $r\ge0$ is the interest rate,
$\delta_i\ge0$, $i=1,\dots,d$, are dividend rates and
$\xi=(\xi^1,\dots,\xi^d)$ is a $d$-dimensional L\'evy process such
that $\xi_0=0$. We assume that $\xi$ has a nondegenerate Gaussian
component and that its L\'evy measure $\nu$ satisfies some natural
integrability conditions.  Note that, in case that $\nu\equiv0$,  our
model reduces to the classical multidimensional Black and Scholes
model with dividend-paying assets.

Let $\psi:\BR^d\rightarrow\BR_+$ be a continuous function with
polynomial growth. Under the fixed risk-neutral measure $P$, in  the
L\'evy model (\ref{eq1.1}), the value at time $s$ of the European
option with payoff function $\psi$ and expiration time $T$ is
given by
\[
V^E(s,x)=Ee^{-r(T-s)}\psi(X^{s,x}_T),
\]
while the value of an American option is given by
\begin{equation}
\label{eq1.2} V(s,x)=\sup_{\tau\in\TT_{s,T}}
Ee^{-r(\tau-s)}\psi(X^{s,x}_\tau),
\end{equation}
where the supremum is taken over the set $\TT_{s,T}$ of all
stopping times (with respect to the filtration generated by $\xi$)
with values in $[s,T]$. It is known (see Pham (1998)) and also
Cont and Tankov (2004), and  Reich, Schwab and Winter (2010)) that, for some assumptions of $\psi,\nu$, the value
function $V$ can be characterized as the unique viscosity solution
of the obstacle problem (or, in another terminology,
integro-differential variational inequality) of the form
\begin{equation}
\label{eq1.3} \min\{-\partial_s V-L u+rV,V-\psi\}=0,\quad
V(T)=\psi,
\end{equation}
where $L$ is the infinitesimal generator of the process $X^{s,x}$.
The main purpose of this paper  is to study two different, but, as
we shall see, closely related goals. The first  is to
carefully examine the relation between the stopping problem
(\ref{eq1.2}) and the Sobolev solutions of (\ref{eq1.3}). In
particular, the problem is to investigate the regularity of the
solution of (\ref{eq1.3}). The second goal is to derive the early
exercise premium formula, i.e., a formula for the difference
$V-V^E$.

In the case of the multidimensional Black and Scholes model, these
problems are quite well investigated (see
Jaillet, Lamberton and Lapeyre (1990), Broadie and Detemple (1997), Villeneuve (1999), Detemple, Feng and Tian (2003), Laurence and Salsa (2009),  Klimsiak and Rozkosz (2016), and the monograph by Detemple (2006)). In the case of $\nu\neq0$, the situation is different.  Although the valuation of
American options in the exponential L\'evy model has been a subject of
numerous investigations (see, e.g., Pham (1997, 1998), Gukhal (2001), Lamberton and Mikou (2008, 2013), Reich et. al. (2010), and the monograph by Cont and Tankov (2004); for numerical methods see, e.g.,
Cont and Tankov (2004), Hilber, Reich, Schwab and Winter (2009), and Matache, von Petersdorff and Schwab  (2004)),  relatively little is known about regularity of $V$, and no general  formula for $V-V^E$ is known, even in the case of $d=1$. Partial results in this direction were obtained in
Pham (1997), Gukhal (2001) and Lamberton and Mikou (2008, 2013) in the case that $d=1$. In particular, in Lamberton and Mikou (2008),
it is shown that  the value of the American put satisfies
(\ref{eq1.3}) in the sense of distributions, and in Lamberton and Mikou (2013), the exercise premium formula is derived. In Pham (1997) an exercise
premium formula for American put is derived by using the theory of
the viscosity solutions of (\ref{eq1.3}).

In the present paper, we consider the Sobolev space solutions  of
(\ref{eq1.3}). From the general theory of variational inequalities,
it follows that (\ref{eq1.3}) has a variational solution $u$ in
the space $\WW^{0,1}_{\varrho}$ with some weight $\varrho$
depending on $\psi$ (for the definitions of various Sobolev spaces,
see Section \ref{sec4.1}). To obtain better regularity of $u$,  we
regard (\ref{eq1.3}) as a complementarity problem (see Bally, Caballero, Fernandez and El Karoui (2002), and Kinderlehrer and Stampacchia (1980)). This means that, by a solution of (\ref{eq1.3}), we
mean a pair $(V,\mu)$ consisting of $V\in\WW^{0,1}_{\varrho}\cap
C([0,T]\times\BR^d)$ and a Radon measure $\mu$ on
$Q_T=[0,T]\times\BR^d$ such that
\begin{equation}
\label{eq1.4} V(T)=\psi,\quad u\ge\psi,\quad
\int_{Q_T}(V-\psi)\varrho^2\,d\mu=0
\end{equation}
and the equation
\begin{equation}
\label{eq1.5} \partial_su+L V=rV-\mu
\end{equation}
is satisfied in the strong sense. Our main result says that, for a
broad class of payoff functions $\psi$, the measure $\mu$ is
absolutely continuous with respect to the Lebesgue measure and that its
density $g$ is square integrable with weight $\varrho^2$. This
shows that, in fact, $V$ satisfies (\ref{eq1.4}), (\ref{eq1.5}) with
$\mu$ replaced by $g$, which  allows us to use results on the regularity
of solutions of the Cauchy problem to show that $u\in
W^{1,2}_{\varrho}$ (in fact, our results on the Cauchy problem
consist of suitable modification of the classical results of
Bensoussan and Lions (1982)). We also compute a formula for $g$. Roughly speaking,
this formula can be translated into the exercise premium formula. Our
exercise premium formula considerably generalizes the results of Lamberton and Mikou (2008, 2013) (note, however, that, in these papers,  the case
with no Gaussian component is also considered). Alternately, it
generalizes the formula proved in Klimsiak and Rozkosz (2016) in the setting of the multidimensional Black and Scholes model.

The proof of our main results  relies on careful analysis of
the reflected backward backward stochastic differential equation
associated with the problem (\ref{eq1.4}), (\ref{eq1.5}). This general
idea comes from  Klimisk and Rozkosz (2011, 2016).

\section{Exponential L\'evy model}
\label{sec2}

Let $\xi=\{\xi_t:t\ge0\}$ be a $d$-dimensional L\'evy process with
generating triplet $(a,\nu,\gamma)$, i.e., a stochastically
continuous  c\`adl\`ag stochastic process with independent and
stationary increments such that $\xi_0=0$, and for $t>0$, the
characteristic function of $\xi_t$ has the following
L\'evy-Khintchine representation
\[
Ee^{i(z,\xi_t)}=e^{t\phi(z)},\quad z\in\BR^d,
\]
where
\[
\phi(z)=-\frac12(z,az)+i(\gamma,z)
+\int_{\BR^d}(e^{i(z,y)}-1-i(z,y)\fch_{\{|y|\le1\}})\,\nu(dy)
\]
(see, e.g., Sato (1999)). In the above formula  $a$ is a symmetric
nonnegative definite $d\times d$ matrix, $\gamma\in\BR^d$ and
$\nu$ is a Borel measure on $\BR^d$ such that $\nu(\{0\})=0$ and
$\int_{\BR^d}(1\wedge |x|^2)\,\nu(dx)<\infty$.

In this paper, we assume that, under the risk-neutral measure $P$ (generally non-unique), the prices $X^{s,x,1},\dots,X^{s,x,d}$ of
financial assets  on the time interval $[s,T]$ are modeled by
(\ref{eq1.1}) with $\xi$ being a L\'evy process under $P$. This
means that, in particular,  if $\delta_i=0$, $i=1,\dots,d$, then
under $P$ the discounted prices $t\mapsto
e^{-r(t-s)}X^{s,x,i}_t=e^{\xi^i_t-\xi^i_s}$, $i=1,\dots,d$, are
martingales under $P$. It is known (see, e.g., Reich et al. (2010, lemma
2.1)) that the last requirement is equivalent to the
following conditions on the triplet $(a,\nu,\gamma)$
\begin{equation}
\label{eq2.1} \int_{\{|y|>1\}}e^{y_i}\,\nu(dy)<\infty,\quad
i=1,\dots,d
\end{equation}
and
\begin{equation}
\label{eq2.2}
\gamma_i+\frac12a_{ii}+\int_{\BR^d}
(e^{y_i}-1-y_i\fch_{\{|y|<1\}})\nu(dy)=0,
\quad i=1,\dots,d.
\end{equation}
We will also assume that
\begin{equation}
\label{eq2.02} \det a>0.
\end{equation}

By  It\^o's formula, under the measure $P$ we have
\begin{align}
\label{eq2.3}
X^{s,x,i}_t-x_i&=\int^t_s(r-\delta_i)X^{s,x,i}_{\theta}\,d\theta
+\int^t_sX^{s,x,i}_{\theta-}\,d\xi^i_{\theta}
+\frac12\int^t_sX^{s,x,i}_{\theta}\,d[\xi^i]^c_{\theta}\\
&\quad+\sum_{s<\theta\le t}
\{X^{s,x,i}_{\theta-}(e^{\Delta\xi^i_{\theta}}-1)
-X^{s,x,i}_{\theta-}\Delta\xi^i_{\theta}\}.\nonumber
\end{align}
Let $J$ denote the Poisson random measure on
$\BR_+\times(\BR^d\setminus\{0\})$ with intensity $\nu$ and let
$\tilde J(dt,dy)=J(dt,dy)-dt\,\nu(dy)$. By the L\'evy-It\^o
decomposition (see, e.g., Protter (2004, theorem I.42) or Sato (1999,
section 19)), for $i=1,\dots,d$ we have
\[
\xi^i_t=\xi^i_s+\sum^d_{i,j=1}\int^t_s\sigma_{ij}\,dW^j_{\theta}
+\int^t_s\gamma_i\,d\theta+\int^t_s\!\int_{\{|y|<1\}}y_i\tilde
J(d\theta,dy)+\sum_{s<\theta\le t}
\Delta\xi^i_{\theta}\fch_{\{|\Delta\xi_{\theta}|\ge1\}},
\]
where $\sigma$ is a $d\times d$-matrix such that
$\sigma\sigma^*=a$ and  $(W^1,\dots,W^d)$ is a standard
$d$-dimensional Wiener process. Using this and (\ref{eq2.2}), one
can show by direct computation that,  from (\ref{eq2.3}), it follows
that $X^{s,x,i}$ is a solution of the equation
\begin{align}
\label{eq2.03}
X^{s,x,i}_t&=x_i+\int^t_s(r-\delta_i)X^{s,x,i}_{\theta}\,d\theta
+\sum^d_{j=1}\int^t_s\sigma_{ij}X^{s,x,i}_{\theta}\,dW^j_{\theta}
\\
&\quad+\int^t_s\!\!\int_{\BR^d}X^{s,x,i}_{\theta-}
(e^{y_i}-1)\tilde J(d\theta,dy)\nonumber\\
&=x_i+\int^t_s(r-\delta_i)X^{s,x,i}_{\theta}\,d\theta
+\int^t_sd(M^{c,i}_{\theta}+M^{d,i}_{\theta})\nonumber
\end{align}
with
\[
M^{c,i}_t=\int^t_s\sigma_{ij}X^{s,x,i}_{\theta}\,dW^j_{\theta},
\quad M^{d,i}_t=\int^t_s\!\!\int_{\BR^d}X^{s,x,i}_{\theta-}
(e^{y_i}-1)\tilde J(d\theta,dy),\quad t\ge s.
\]

Let $X^{s,x}=(X^{s,x,1},\dots,X^{s,x,d})$ be the process defined
by (\ref{eq1.1}) and let $P_{s,t}$ denote its transition function,
i.e., $P_{s,t}(x,B)=P(X^{s,x}_t\in B)$ for all $t>s$ and Borel
set $B\subset\BR^d$. Of course, $P_{s+h,t+h}(x,B)=P_{s,t}(x,B)$
for $h\ge0$.

In what follows  by  $\BX^s=((X_t)_{t\ge s}\, ,(\FF^s_t)_{s\ge t}\,,
(P_{s,x})_{x\in\BR^d})$ we denote a temporally homogeneous Markov
process with transition function $P_t(x,B)=P(X^{s,x}_t\in B)$,
$t>s$. With this notation, the law of $X^{s,x}$ under $P$ is the
same as $X$ under $P_{s,x}$. By $E_{s,x}$ we denote the
expectation with respect to $P_{s,x}$.

Let $I=\{0,1\}^{d}$. Set
\[
D_{\iota}=\{x\in\BR^d: (-1)^{i_{k}}x_{k}>0,k=1,\dots,d\}\mbox{ for
}\iota=(i_1,\dots,i_d)\in I,\quad D=\bigcup_{\iota\in I}
D_{\iota}.
\]

\begin{remark}
\label{rem2.1} (i) Let $x\in D_{\iota}$ for some $\iota\in I$.
Then, from (\ref{eq1.1}), it immediately follows that $P_{s,x}(X_t\in
D_{\iota},t\ge s)=1$ for every $s\ge0$,
\smallskip\\
(ii) If (\ref{eq2.02}) is satisfied,  then for all $t>0$ and
$x\in D_{\iota}$, the distribution of $X_t$ under $P_{0,x}$ is
absolutely continuous. Let $p(t,x,y)$ denote its density. Then,
$(0,T)\times D_{\iota}\times D_{\iota}\ni(t,x)\mapsto p(t,x,y)$ is
strictly positive and continuous. To see this, let us first note
that, by Sato (1999, theorem 19.2(iii)), for every $t>0$ and
$x\in\BR^d$, the distribution of the random variable $x+\xi_t$ is
equal  to the convolution of the Gaussian measure $\NN(x+\gamma
t,at)$ and the distribution $\mu_t$ of $\bar\xi_t$, where
$\bar\xi$ is a L\'evy process with the characteristic triplet
$(0,\nu,0)$. Therefore, the distribution of $x+\xi_t$ has density
of the form
\begin{equation}
\label{eq2.5} q(t,x,y)=\int_{\BR^d}g_x(t,y+z)\mu_t(dz),\quad
y\in\BR^d\, ,
\end{equation}
where $g_x(t,y)$ denotes the density of the measure  $\NN(x+\gamma
t,at)$. From (\ref{eq2.5}), it immediately follows that
$q(\cdot,x,\cdot)$ is strictly positive on $(0,T)\times\BR^d$.
Using (\ref{eq2.5}) and performing elementary calculations, one may
also show that $(0,T)\times\BR^d\ni(t,y)\mapsto q(t,x,y)$ is
continuous. The desired properties of $p$ now follow from
(\ref{eq1.1}).
\end{remark}

\section{Optimal stopping problem and reflected BSDEs}
\label{sec3}

In this paper, we assume that $\psi:\BR^d\rightarrow\BR_+$ is a
measurable function such that
\begin{equation} \label{eq3.1}
\psi(x)\le K(1+|x|^p),\quad x\in\BR^d
\end{equation}
for some $K\ge0$, $p\ge0$.
As for $\nu$, in this section, we assume that, for some
$\varepsilon>0$,
\begin{equation}
\label{eq2.06} \int_{\{|y|>1\}}e^{((1\vee p)+\varepsilon)y_i}
\,\nu(dy)<\infty, \quad i=1,\dots,d.
\end{equation}
By Sato (1999, theorem 25.3), the condition (\ref{eq2.06}) implies that
$E_{s,x}|X_T|^{(1\vee p)+\varepsilon}<\infty$ for every $(s,x)\in
Q_T=[0,T]\times\BR^d$. In particular, if $\psi,\nu$ satisfy
(\ref{eq3.1}), (\ref{eq2.06}), then $E_{s,x}\psi(X_T)<\infty$ for
$(s,x)\in Q_T$.

The value at time $t\in[s,T]$ of the American option with terminal
payoff $\psi(X_T)$  is given by
\begin{equation}
\label{eq2.08} V_t=\mbox{ess\,sup}_{\tau\in\TT_{t,T}}
E_{s,x}\big(e^{-r(\tau-s)}\psi(X_{\tau})|\FF^s_t\big),
\end{equation}
where $E_{s,x}$ denotes the expectation with respect to $P_{s,x}$
and $\TT_{s,T}$ is the set of all $(\FF^s_t)$-stopping times with
values in $[s,T]$. It is known that
\begin{equation}
\label{eq2.10} V_t=u(t,X_t),\quad t\in[s,T],\quad
P_{s,x}\mbox{-a.s.},
\end{equation}
where
\begin{equation}
\label{eq2.09} u(s,x)=\sup_{\tau\in\TT_{s,T}}
E_{s,x}e^{-r(\tau-s)}\psi(X_{\tau}) \quad (=V(s,x)).
\end{equation}

The optimal stopping problem (\ref{eq2.08}) is closely related to
the solution of some reflected backward stochastic differential equation (reflected BSDE). To state the relation, let
us first recall  that a triple $(Y^{s,x},M^{s,x},K^{s,x})$
consisting of a c\`adl\`ag $(\FF^s_t)$-adapted process $Y^{s,x}$
of class D, a c\`adl\`ag $((\FF^s_t),P_{s,x})$-local martingale
$M^{s,x}$ such that $M^{s,x}_s=0$ and a c\`adl\`ag
$(\FF^s_t)$-predictable increasing process $K^{s,x}$ such that
$K^{s,x}_s=0$ is a solution, on the filtered probability space
$(\Omega,(\FF^s_t),P_{s,x})$, of the reflected BSDE
\begin{equation}
\label{eq3.06}
Y^{s,x}_t=\psi(X_T)-\int^T_trY^{s,x}_{\theta}\,d\theta
+\int^T_tdK^{s,x}_{\theta} -\int^T_tdM^{s,x}_{\theta},\quad
t\in[s,T]
\end{equation}
with barrier $\psi(X)$ if
\[
Y^{s,x}_t\ge \psi(X_t),\quad t\in[s,T],\quad
\int^T_s(Y^{s,x}_{t-}-\psi(X_{t-}))\,dK^{s,x}_t=0,\quad
P_{s,x}\mbox{-a.s.}
\]
and (\ref{eq3.06}) is satisfied $P_{s,x}$. Let us observe that, if
the restriction $\psi_{|D}$ of $\psi$ to $D$ is continuous, then
by Remark \ref{rem2.1}(ii), the barier $\psi(X)$ is a c\`adl\`ag
process under $P_{s,x}$ for every $(s,x)\in[0,T)\times D$.


\begin{theorem}
\label{th3.1} Assume that $\psi$  satisfies
\mbox{\rm(\ref{eq3.1})} and  that $\psi_{|D}$ is continuous, $\nu$
satisfies \mbox{\rm(\ref{eq2.06})}, and let $(s,x)\in[0,T)\times
D$.
\begin{enumerate}
\item[\rm(i)]There exists a
unique solution $(Y^{s,x},M^{s,x},K^{s,x})$ of
\mbox{\rm(\ref{eq3.06})}. Moreover, $M^{s,x}$ is an
$((\FF^s_t),P_{s,x})$-uniformly integrable martingale, $K^{s,x}$
is continuous and $E_{s,x}K^{s,x}_T<\infty$.

\item[\rm(ii)]$V_t=Y^{s,x}_t$, $t\in[s,T]$, $P_{s,x}$-a.s. Hence,
if we define $u:[0,T]\times D\rightarrow\BR_+$ by
\mbox{\rm(\ref{eq2.09})}, then
\begin{equation}
\label{eq2.15} Y^{s,x}_t=u(t,X_t),\quad t\in[s,T],\quad
P_{s,x}\mbox{-a.s.}
\end{equation}
Moreover, $u$ is continuous.

\item[\rm(iii)]For every $s\in[0,T)$, there exists a continuous
process $K^s$ on $[s,T]$ such that $K^s$ is
$P_{s,x}$-indistinguishable from $K^{s,x}$ for every $x\in D$, and
for every $s\in[0,T)$, there exists a martingale  $M^s$ on $[s,T]$
such that $M^s$ is $P_{s,x}$-indistinguishable from $M^{s,x}$ for
every $x\in D$.
\end{enumerate}
\end{theorem}
\begin{proof}
As $E_{s,x}\psi(X_T)<\infty$ and the filtration $(\FF^s_t)$ is
quasi-left continuous (see Protter (2004, exercise III.9)), the
existence and uniqueness of a solution $(Y^{s,x},M^{s,x},K^{s,x})$
of (\ref{eq3.06}) such that $E_{s,x}K^{s,x}_T<\infty$ follows from
Corollary 2.2 and Theorem 2.13  in Klimsiak (2015). Moreover,
$M^{s,x}$ is uniformly integrable (see the remark following eq.
(2.28) in Klimsiak (2015)). Set
\[
\bar Y_t=e^{-r(t-s)}Y^{s,x}_t,\quad \bar
M_t=\int^t_se^{-r(\theta-s)}\,dM^{s,x}_{\theta},\quad \bar
K_t=\int^t_se^{-r(\theta-s)}\,dK^{s,x}_{\theta},\quad t\in[s,T].
\]
By integrating by parts,  one can check that the triple $(\bar Y,\bar
M,\bar K)$ is a solution of the reflected BSDE
\begin{equation}
\label{eq2.11} \bar Y_t=\bar\xi+\int^T_td\bar K_{\theta}
-\int^T_td\bar M_{\theta},\quad t\in[s,T],\quad
P_{s,x}\mbox{-a.s.}
\end{equation}
with $\bar\xi=e^{-r(T-s)}\psi(X_T)$ and barrier $\bar
L_t=e^{-r(t-s)}\psi(X_t)$, $t\in[s,T]$. Therefore, from
Klimsiak (2015, corollary 2.9) (with $f=0$, $V=0$, $\hat L=\bar L$),
it follows that
\[
e^{-r(t-s)}Y^{s,x}_t=\bar Y_t=\mbox{ess\,sup}_{\tau\in\TT_{s,T}}
E_{s,x}\big(e^{-r(\tau-s)}\psi(X_{\tau})|\FF^s_t\big),
\]
which implies  the first part of (ii). The second part of (ii) now
follows from (\ref{eq2.10}).
Now, we are going to show that $V$ defined by (\ref{eq1.2}) is
continuous and hence that $u$ is continuous. By Corollary 2.2 and
Theorem 2.13 in Klimisak (2015), there exists a unique solution
$(\tilde Y^{s,x},\tilde M^{s,x},\tilde K^{s,x})$ of the reflected
BSDE
\[
\tilde Y^{s,x}_t=\eta^{s,x}+\int^T_td\tilde K^{s,x}_{\theta}
-\int^T_td\tilde M^{s,x}_{\theta},\quad t\in[s,T],\quad
P\mbox{-a.s.}
\]
with terminal condition $\eta^{s,x}=e^{-r(T-s)}\psi(X^{s,x}_T)$
and barrier $L^{s,x}_t=e^{-r(t-s)}\psi(X^{s,x}_t)$, $t\in[s,T]$.
In what follows, we extend $X^{s,x}$ and $L^{s,x}$ to $[0,T]$ by
putting $X^{s,x}_t=x$ for $t\in[0,s]$. Suppose that $x\in
D_{\iota}$ for some $\iota\in I$. Fix $q\in(1,1+(\varepsilon/p))$
and consider sequences $\{s_n\}\subset[0,T]$, $\{x_n\}\subset
D_{\iota}$ such that $s_n\rightarrow s$, $x_n\rightarrow x$. By
Klimisak (2015, proposition 5.1),
\begin{align}
\label{eq2.14} &|\tilde Y^{s_n,x_n}_t-\tilde Y^{s,x}_t|^q \le
|\tilde Y^{s_n,x_n}_T-\tilde Y^{s,x}_T|^q \\
&\qquad+q\int^T_t |\tilde Y^{s_n,x_n}_{\theta-}-\tilde
Y^{s,x}_{\theta-}|^{q-1}\mbox{\rm sign}(\tilde
Y^{s_n,x_n}_{\theta-}-\tilde Y^{s,x}_{\theta-})\,d(\tilde
K^{s_n,x_n}_{\theta}-\tilde K^{s,x}_{\theta})\nonumber\\
&\qquad-q\int^T_t |\tilde Y^{s_n,x_n}_{\theta-}-\tilde
Y^{s,x}_{\theta-}|^{q-1} \mbox{\rm sign}(\tilde
Y^{s_n,x_n}_{\theta-}-\tilde Y^{s,x}_{\theta-})\,d(\tilde
M^{s_n,x_n}_{\theta}-\tilde M^{s,x}_{\theta}),\nonumber
\end{align}
where $\mbox{\rm sign}(x)=1$ if $x>0$ and  $\mbox{\rm sign}(x)=-1$
if $x\le0$. 
We have
\begin{align*}
I^n_t&:=\int^T_t\fch_{\{\tilde Y^{s_n,x_n}_{\theta-}>\tilde
Y^{s,x}_{\theta-}\}}|\tilde Y^{s_n,x_n}_{\theta-}-\tilde
Y^{s,x}_{\theta-}|^{q-1} \mbox{\rm sign}(\tilde
Y^{s_n,x_n}_{\theta-}-\tilde Y^{s,x}_{\theta-})\,d(\tilde
K^{s_n,x_n}_{\theta}-\tilde K^{s,x}_{\theta})\\
&\le \int^T_t\fch_{\{\tilde Y^{s_n,x_n}_{\theta-}>\tilde
Y^{s,x}_{\theta-}\}}\fch_{\{L^{s_n,x_n}_{\theta-}\le
L^{s,x}_{\theta-}\}} |\tilde Y^{s_n,x_n}_{\theta-}-\tilde
Y^{s,x}_{\theta-}|^{q-1} \frac{\tilde Y^{s_n,x_n}_{\theta-}-\tilde
Y^{s,x}_{\theta-}\wedge \tilde Y^{s_n,x_n}_{\theta-}} {|\tilde
Y^{s_n,x_n}_{\theta-}-\tilde Y^{s,x}_{\theta-}|}
\,d\tilde K^{s_n,x_n}_{\theta}\\
&\quad+\int^T_t\fch_{\{\tilde Y^{s_n,x_n}_{\theta-}>\tilde
Y^{s,x}_{\theta-}\}}\fch_{\{L^{s_n,x_n}_{\theta-}>L^{s,x}_{\theta-}\}}
|\tilde Y^{s_n,x_n}_{\theta-}-\tilde Y^{s,x}_{\theta-}|^{q-1}
\,d\tilde K^{s_n,x_n}_{\theta}=:I^{n,1}_t+I^{n,2}_t.
\end{align*}
Observe that
\[
I^{n,1}_t\le \int^T_t\fch_{\{\tilde Y^{s_n,x_n}_{\theta-}>\tilde
Y^{s,x}_{\theta-}\}} |\tilde Y^{s_n,x_n}_{\theta-}-\tilde
Y^{s,x}_{\theta-}|^{q-1} \frac{\tilde
Y^{s_n,x_n}_{\theta-}-L^{s_n,x_n}_{\theta-}} {|\tilde
Y^{s_n,x_n}_{\theta-}-\tilde Y^{s,x}_{\theta-}|} \,d\tilde
K^{s_n,x_n}_{\theta}=0.
\]
As $0\le\tilde  Y^{s_n,x_n}_{t-}-\tilde Y^{s,x}_{t-}\le \tilde
Y^{s_n,x_n}_{t-}-L^{s_n,x_n}_{t-}+L^{s_n,x_n}_{t-}-L^{s,x}_{t-}$
if $\tilde Y^{s_n,x_n}_{t-}>\tilde Y^{s,x}_{t-}$, we have
\begin{align*}
I^{n,2}_t&\le 2^{q-1} \int^T_t\fch_{\{\tilde
Y^{s_n,x_n}_{\theta-}>\tilde Y^{s,x}_{\theta-}\}}
\fch_{\{L^{s_n,x_n}_{\theta-}>L^{s,x}_{\theta-}\}} \\
&\qquad\qquad\qquad\qquad\quad\times((\tilde
Y^{s_n,x_n}_{\theta-}-L^{s_n,x_n}_{\theta-})^{q-1}
+(L^{s_n,x_n}_{\theta-}-L^{s,x}_{\theta-})^{q-1})
\,d\tilde K^{s_n,x_n}_{\theta}\\
&\le 2^{q-1} \int^T_t\fch_{\{\tilde Y^{s_n,x_n}_{\theta-}>\tilde
Y^{s,x}_{\theta-}\}} (\tilde
Y^{s_n,x_n}_{\theta-}-L^{s_n,x_n}_{\theta-})^{q-2} (\tilde
Y^{s_n,x_n}_{\theta-}-L^{s_n,x_n}_{\theta-})\,d\tilde K^{s_n,x_n}_{\theta}\\
&\quad+2^{q-1}\sup_{0\le t\le
T}|L^{s_n,x_n}_t-L^{s,x}_t|^{q-1}\tilde K^{s_n,x_n}_T.
\end{align*}
Because the first integral on the right-hand side of the above
inequality is  equal to zero, combining the estimates for
$I^{n,1}$ and $I^{n,2}$ yields
\begin{equation}
\label{eq2.18} I^n_t\le 2^{q-1}\sup_{0\le t\le
T}|L^{s_n,x_n}_t-L^{s,x}_t|^{q-1}K^{s_n,x_n}_T.
\end{equation}
In much the same manner as above, one can prove that
\begin{align}
\label{eq2.19} &\int^T_t\fch_{\{\tilde
Y^{s_n,x_n}_{\theta-}\le\tilde Y^{s,x}_{\theta-}\}}|\tilde
Y^{s_n,x_n}_{\theta}-\tilde Y^{s,x}_{\theta}|^{q-1} \mbox{\rm
sign}(\tilde Y^{s_n,x_n}_{\theta-}-\tilde
Y^{s,x}_{\theta-})\,d(\tilde
K^{s_n,x_n}_{\theta}-\tilde K^{s,x}_{\theta})\\
&\qquad\le 2^{q-1}\sup_{0\le t\le
T}|L^{s_n,x_n}_t-L^{s,x}_t|^{q-1}K^{s,x}_T. \nonumber
\end{align}
Let  $(\tilde\FF^s_t)$ denote the usual augmentation of the
filtration generated by $X^{s,x}$. By (\ref{eq2.14}),
(\ref{eq2.18}) and (\ref{eq2.19}),
\begin{align*}
|\tilde Y^{s_n,x_n}_t-\tilde Y^{s,x}_t|^q &=E(|\tilde
Y^{s_n,x_n}_t-\tilde Y^{s,x}_t|^q|\tilde\FF^s_t) \\
&\le E\big(|\eta^{s_n,x_n}-\eta^{s,x}|^q+2^q\sup_{0\le t\le
T}|L^{s_n,x_n}_t-L^{s,x}_t|^{q-1}(\tilde K^{s_n,x_n}_T+\tilde
K^{s,x}_T) |\tilde\FF^s_t\big).
\end{align*}
By the above inequality and Briand et al. (2003, lemma 6.1),
\begin{align}
\label{eq2.24} &E(\sup_{0\le t\le T}|\tilde Y^{s_n,x_n}_t-\tilde
Y^{s,x}_t|^{q/2})\\
&\quad \le 2 \big(E|\eta^{s_n,x_n}-\eta^{s,x}|^q  + 2^qE\sup_{0\le
t\le T} |L^{s_n,x_n}_t-L^{s,x}_t|^{q-1}
(\tilde K^{s_n,x_n}_T+\tilde K^{s,x}_T)\big)^{1/2}\nonumber\\
&\quad \le 2 \big(E|\eta^{s_n,x_n}-\eta^{s,x}|^q \nonumber\\
&\qquad\qquad+2^q(E\sup_{0\le t\le T}
|L^{s_n,x_n}_t-L^{s,x}_t|^{q})^{(q-1)/q} (E(\tilde
K^{s_n,x_n}_T+\tilde K^{s,x}_T)^q)^{1/q}\big)^{1/2}.\nonumber
\end{align}
Our next claim is that
\begin{equation}
\label{eq2.21} \lim_{n\rightarrow\infty}E\sup_{0\le t\le T}
|X^{s_n,x_n}_t-X^{s,x}_t|^q=0
\end{equation}
and
\begin{equation}
\label{eq2.20} \sup_{n\ge1}E(K^{s_n,x_n}_T+K^{s,x}_T)^q<\infty.
\end{equation}
To prove (\ref{eq2.21}), let us first observe that for every
$x,y\in\BR^d$ and $i=1,\dots,d$,
\begin{align}
\label{eq2.22} E\sup_{0\le t\le T}|X^{s,x,i}_t-X^{s,y,i}_t|^q \le
C|x^i-y^i|^q
\end{align}
for some $C>0$ depending only on $T,r,\delta,q$ and $\nu$. Indeed,
$|X^{s,x,i}_t-X^{s,y,i}_t|^q=0$ for $t\in[0,s]$. Furthermore,
because $[s,T]\ni t\mapsto e^{\xi^i_t-\xi^i_s}$ is a martingale
under $P$, it follows from (\ref{eq1.1}) and Doob's inequality
that
\begin{align*}
E\sup_{s\le t\le T}|X^{s,x,i}_t-X^{s,y,i}_t|^q
&\le|x^i-y^i|^qe^{q|r-\delta_i|(T-s)} E|\sup_{s\le t\le T}
e^{\xi^i_t-\xi^i_s}|^q\nonumber\\
&\le|x^i-y^i|^q (\frac{q}{q-1})^qe^{q|r-\delta_i|(T-s)}
Ee^{q(\xi^i_T-\xi^i_s)},
\end{align*}
which when combined with  (\ref{eq2.06}) and Sato (1999, theorem 25.3) yields
(\ref{eq2.22}). Furthermore,
\begin{equation}
\label{eq2.23} \lim_{h\rightarrow0}E\sup_{0\le s\le T}
|X^{s+h,x}_t-X^{s,x}_t|^q=0.
\end{equation}
Indeed, if $h\ge0$ and $s+h\le T$, then
\begin{align*}
&\sup_{s+h\le t\le T}|X^{s+h,x,i}_t-X^{s,x,i}_t|^q \\
&\qquad =|x^i|^q \sup_{s+h\le t\le T}
e^{q(r-\delta_i)(t-s)+q(\xi^i_t-\xi^i_s)}
|e^{-(r-\delta_i)h+\xi^i_s-\xi^i_{s+h}}-1|^q\\
&\qquad\le|x^i|^q
e^{q|r-\delta_i|(T-s)}|e^{-(r-\delta_i)h+\xi^i_s-\xi^i_{s+h}}-1|^q
\sup_{s+h\le t\le T}| e^{\xi^i_t-\xi^i_s}|^q.
\end{align*}
By Doob's inequality and (\ref{eq2.06}), $E\sup_{s\le t\le T}|
e^{\xi^i_t-\xi^i_s}|^q<\infty$. From this and the fact that
$|e^{-(r-\delta_i)h+\xi^i_s-\xi^i_{s+h}}-1|^q\rightarrow0$ in
probability $P$ as $h\rightarrow0$, it follows that the right-hand
side of the above inequality converges in probability $P$ to zero
as $h\rightarrow0$. A similar argument shows that $\sup_{s\le t\le
s+h}|X^{s+h,x,i}_t-X^{s,x,i}_t|^q\rightarrow$ in probability $P$,
and hence that
\begin{equation}
\label{eq2.25} I^{s,h}:=\sup_{0\le t\le T}
|X^{s+h,x,i}_t-X^{s,x,i}_t|^q\rightarrow0\quad\mbox{in probability
}P
\end{equation}
as $h\rightarrow0^+$. In the same manner, we can see that
(\ref{eq2.25}) holds true if $h\rightarrow0^{-}$. Using Doob's
inequality and (\ref{eq2.06}), one can also show that for each
fixed $s\in[0,T]$, $\sup_{h}E|I^{s,h}|^{\alpha}<\infty$ for some
$\alpha>0$, so for each $s\in[0,T]$, the family $\{I^{s,h}\}$ is
uniformly integrable. This and (\ref{eq2.25}) imply
(\ref{eq2.23}). Combining (\ref{eq2.22}), (\ref{eq2.23}) with the
inequality
\[
\sup_{0\le t\le T} |X^{s_n,x_n}_t-X^{s,x}_t|^q \le2^{q-1}
(\sup_{0\le t\le T} |X^{s_n,x_n}_t-X^{s_n,x}_t|^q +\sup_{0\le t\le
T} |X^{s_n,x}_t-X^{s,x}_t|^q)
\]
we get (\ref{eq2.21}). By Klimsiak (2015, proposition 5.4), for every
$(s,x)\in[0,T)\times\BR^d$ and $k\ge0$, there exists a unique
solution $(\tilde Y^k,\tilde M^k)$ of the BSDE
\[
\tilde Y^k_t=\eta^{s,x}+k\int^T_t(\tilde
Y^k_{\theta}-L^{s,x}_{\theta})^{-}\,d\theta -\int^T_td\tilde
M^k_{\theta},\quad t\in[s,T].
\]
Moreover, from the  proof of Proposition 5.4 in Klimsiak (2015), it
follows  that there is $C$ not depending on $k,s,x$ such that
\[
E\sup_{s\le t\le T}|\tilde Y^{k}_t|^q\le CE\Big(|\eta^{s,x}|^q
+(\int^T_s|\psi(X^{s,x}_{\theta})|\,d\theta)^q\Big).
\]
As $\psi$ satisfies  (\ref{eq3.1}), it follows from the above
estimate, Doob's inequality and (\ref{eq2.06}) that $E\sup_{s\le
t\le T}|\tilde Y^{k}_t|^q\le C_1E\sup_{s\le t\le
T}|X^{s,x}_t|^{qp}\le C_2$ for some constants $C_1,C_2$ not
depending on $k,s,x$. Because by Klimsiak (2015, theorem 2.13), $\tilde
Y^0_t\le\tilde Y^k_t\nearrow\tilde Y^{s,x}_t$, $t\in[s,T]$, as
$k\nearrow\infty$, applying Fatou's gives $E\sup_{s\le t\le
T}|\tilde Y^{s,x}_t|^q\le C_2$. By this and Klimisak (2015, lemma
5.6), $E|K^{s,x}_T|^q\le C_3$ for some $C_3$ not depending
on $s,x$, which proves (\ref{eq2.20}). Observe now that
\begin{equation}
\label{eq2.26}
\lim_{n\rightarrow\infty}E|\eta^{s_n,x_n}-\eta^{s,x}|^q=0,\qquad
\lim_{n\rightarrow\infty}E\sup_{0\le t\le T}
|L^{s_n,x_n}_t-L^{s,x}_t|^{q}=0.
\end{equation}
To see this, for $R>0$ set $A_{n,R}=\{\sup_{s\le t\le
T}(|X^{s_n,x_n}_t|+|X^{s,x}_t|)\le R\}$. As $\psi$ is
continuous, it follows from (\ref{eq2.21}) that
\begin{equation}
\label{eq2.27} \lim_{n\rightarrow\infty}E(\fch_{A_{n,R}}\sup_{0\le
t\le T}|L^{s_n,x_n}_t-L^{s,x}_t|^{q})=0.
\end{equation}
From (\ref{eq2.21}) it also follows that
$\sup_{n\ge1}P(A^c_{n,R})\rightarrow0$ as $R\rightarrow\infty$.
Hence
\begin{equation}
\label{eq2.28} \lim_{R\rightarrow\infty}\sup_{n\ge1}
E(\fch_{A^c_{n,R}}\sup_{0\le t\le T}|L^{s_n,x_n}_t-L^{s,x}_t|^{q})
=0,
\end{equation}
because by (\ref{eq3.1}) and (\ref{eq2.06}),
$\sup_{n\ge1}E\sup_{0\le t\le
T}|L^{s_n,x_n}_t-L^{s,x}_t|^{q_1}<\infty$ for some $q_1>q$. From
(\ref{eq2.27}), (\ref{eq2.28}), we get the second convergence in
(\ref{eq2.26}). In much the same manner, we prove the first
convergence. Combining (\ref{eq2.24}) with (\ref{eq2.20}) and
(\ref{eq2.26}), we see that $E(\sup_{0\le t\le T}|\tilde
Y^{s_n,x_n}_t-\tilde Y^{s,x}_t|^{q/2})\rightarrow0$. We may now
repeat the argument from the beginning of the proof of El Karoui at al. (1997, lemma 8.4) to conclude that $\tilde
Y^{s_n,x_n}_{s_n}\rightarrow\tilde Y^{s,x}_s$. This proves
the continuity of $V$, because  $V(s_n,x_n)=\tilde Y^{s_n,x_n}_{s_n}$
and $V(s,x)=\tilde Y^{s,x}_{s}$. By (\ref{eq3.06}) and
(\ref{eq2.15}),
\begin{equation}
\label{eq2.16}
u(t,X_t)=u(s,X_s)+r\int^t_su(\theta,X_{\theta})\,d\theta
-K^{s,x}_t+M^{s,x}_t,\quad t\in[s,T],\quad P_{s,x}\mbox{-a.s.}
\end{equation}
Because the filtration $(\FF^s_t)$ is quasi-left continuous, the
jump times of the martingale $M^{s,x}$ are all totally
inaccessible (see, e.g., Protter (2004, p. 192)). As the jump times
of the L\'evy process $X$ are also totally inaccessible, it
follows from (\ref{eq2.16}) that the jumps of $K^{s,x}$ can occur
only at totally inaccessible stopping times. Therefore, $K^{s,x}$
is continuous because we know that $K^{s,x}$ is an increasing
predictable process of integrable variation (see, e.g., Corollary to Theorem III.25 in Protter (2004)). For $n\in\BN$, let $(Y^n,M^n)$
denote a solution of the BSDE
\begin{equation}
\label{eq3.36} Y^n_t=\psi(X_T)-\int^T_trY^n_{\theta}\,d\theta
+\int^T_tdK^{s,n}_{\theta}-\int^T_tdM^n_{\theta},\quad t\in[s,T],
\end{equation}
where
$K^{s,n}_t=n\int^t_s(Y^n_{\theta}-\psi(X_{\theta}))^-\,d\theta$,
$t\in[s,T]$. One can show (see, e.g., the proof of Theorem
4.7 in Klimsiak and Rozkosz (2013)) that $Y^n_t=u_n(t,X_t)$, $t\in[s,T]$, where
$u_n(s,x)=Y^n_s$. As $K^{s,x}$ is continuous, it follows from Klimisak (2015, theorem 2.13) that $\sup_{s\le t\le
T}|K^{s,n}_t-K^{s,x}_t|\rightarrow0$ in probability $P_{s,x}$ as
$n\rightarrow\infty$. From this and Fukushima at al. (2011, lemma A.3.4), we
deduce that there is a continuous process $K^s$ on $[s,T]$ such
that $K^s$ is $P_{s,x}$-indistinguishable from $K^{s}$ for every
$x\in\BR^d$. Consequently, by (\ref{eq2.16}), there is a
martingale $M^s$ on $[s,T]$ such that $M^s$ is
$P_{s,x}$-indistinguishable from $M^{s,x}$ for every $x\in\BR^d$.
\end{proof}

\begin{remark}
If $\psi$  is continuous on $\BR^d$ and satisfies (\ref{eq3.1}),
then in the formulation of Theorem \ref{th3.1}, we may replace $D$
by $\BR^d$.
\end{remark}

The following lemma will be needed in the proof of our main result
in Section \ref{sec5}.

\begin{lemma}
\label{lem3.3} For $s\in[0,T)$, let $K^s$ be the process of Theorem
\ref{th3.1}. There exists a unique positive Radon measure $\mu$ on
$(0,T)\times D$ such that, for all $s\in[0,T)$ and $x\in D$,
\begin{equation}
\label{eq3.35} E_{s,x}\int^T_sf(t,X_t)\,dK^s_t
=\int^T_s\!\!\int_{D}f(t,y)p(t-s,x,y)\,d\mu(t,y)
\end{equation}
for every continuous $f:(s,T)\times D\rightarrow\BR$ with compact
support.
\end{lemma}
\begin{proof}
We first prove the existence of $\mu$. Suppose that $x\in
D_{\iota}$ for some $\iota\in I$. Let $u_n, K^n$ be defined as in
the proof of Theorem \ref{th3.1}. Then, for every $f\in
C_c((0,T)\times\BR^d)$,
\begin{equation}
\label{eq3.37} E_{s,x}\int^T_sf(t,X_t)\,dK^{s,n}_t
=\int^T_s\!\!\int_{D}f(t,y)p(t-s,x,y)\,d\mu_n(t,y),
\end{equation}
where $\mu_n=n(u_n(t,y)-\psi(y))^{-}\,dt\,dy$. As $f$ is
bounded and  by Klimsiak (2015, theorem 2.13), $\sup_{s\le t\le T}
|K^{s,n}_t-K^s_t|\rightarrow0$ in probability $P_{s,x}$ and
$E_{s,x}K^{s,n}_{T}\rightarrow E_{s,x}K^s_T$ as
$n\rightarrow\infty$,
$\int^T_sf(t,X_t)\,dK^{s,n}_t\rightarrow\int^T_sf(t,X_t)\,dK^s_t$
in probability $P_{s,x}$ and  $\{\int^T_sf(t,X_t)\,dK^{s,n}_t\}$
is uniformly integrable with respect to $P_{s,x}$. Hence
\begin{equation}
\label{eq3.38} E_{s,x}\int^T_sf(t,X_t)\,dK^{s,n}_t\rightarrow
E_{s,x}\int^T_sf(t,X_t)\,dK^s_t.
\end{equation}
By (\ref{eq3.36}),
\[
e^{-rT}\psi(X_T)-Y^n_0=-\int^T_0e^{-rt}\,dK^{0,n}_t
+\int^T_0e^{-rt}\,dM^n_t.
\]
From this and the fact that $\psi\ge0$ and $Y^n_0=u_n(0,x)$, it
follows that
\[
E_{0,x}K^{0,n}_T\le e^{rT}E_{0,x}\int^T_0e^{-rt}\,dK^{0,n}_t \le
e^{rT}u_n(0,x).
\]
Let $\nu_n=n(u_n(t,y)-\psi(y))^{-}\,p(t,x,y)\,dt\,dy$. Because by
Klimsiak (2015, theorem 2.13), $u_n(0,x)=Y^n_0\le Y^{0,x}_0= u(0,x)$
with $u,Y^{0,x}$ of Theorem \ref{th3.1}, the above inequality
shows that
\[
\sup_{n\ge1}\nu_n((0,T)\times D_{\iota})=\sup_{n\ge1}
E_{0,x}K^{0,n}_T<\infty.
\]
Let $\bar\mu_n$ denote the restriction of $\mu_n$ to $(0,T)\times
D$. Because by Remark \ref{rem2.1}, the function $(0,T)\times
D_{\iota}\ni(t,y)\mapsto p(t,x,y)$ is strictly positive and
continuous, it follows from the above that for every compact set
$K\subset(0,T)\times D_{\iota}$,
$\sup_{n\ge1}\bar\mu_n(K)<\infty$. Because this estimate holds true
for each $\iota\in I$, we in fact have
$\sup_{n\ge1}\bar\mu_n(K)<\infty$ for every compact subset
$K\subset(0,T)\times D$. Therefore there is a subsequence, still
denoted by $n$, such that $\{\bar\mu_n\}$ converges locally
weakly${}^*$ to some positive Radon measure $\mu$ on $(0,T)\times
D$. Consequently,
\begin{equation}
\label{eq3.39}
\int^T_s\!\!\int_{D}f(t,y)p(t-s,x,y)\,d\bar\mu_n(t,y)\rightarrow
\int^T_s\!\!\int_{D}f(t,y)p(t-s,x,y)\,d\mu(t,y).
\end{equation}
Combining (\ref{eq3.37})--(\ref{eq3.39}) proves (\ref{eq3.35}).
Uniqueness of $\mu$ follows easily from the fact that
$p(\cdot,x,\cdot)$ is strictly positive on $(0,T)\times D_{\iota}$
for each $\iota\in I$.
\end{proof}
\medskip

Note that, from Lemma \ref{lem3.3}, it follows in particular that,
for every $x\in D$,
\begin{equation}
\label{eq3.25}
E_{s,x}K^s_T=\int^T_s\!\!\int_{D}p(t-s,x,y)\,d\mu(t,y).
\end{equation}
To see this it suffices to approximate the function
$\fch_{(s,T)\times D}$ by an increasing sequence of positive
continuous functions with compact support and use  monotone
convergence.

\section{Cauchy problem}
\label{sec4}

Let $C_0(\BR^d)$ denote the set of continuous functions on $\BR^d$
vanishing at infinity and let $L$ denote the infinitesimal
generator of the semigroup on $C_0(\BR^d)$ induced by the  process
$X^{s,x}$, i.e.,
\begin{equation}
\label{eq4.1} Lf(x)=L_{BS}f(x)+L_If(x)
\end{equation}
for $f\in C^2_0(\BR^d)$, where
\[
L_{BS}f(x)
=\frac12\sum^d_{i,j=1}a_{ij}x_ix_j\partial^2_{x_ix_j}f(x)
+\sum^d_{i=1}(r-\delta_i)x_i\partial_{x_i}f(x)
\]
and
\[
L_If(x)=\int_{\BR^d}\Big(f(xe^y)-f(x)-
\sum_{i=1}^dx_i(e^{y_i}-1)\partial_{x_i}f(x)\Big)\,\nu(dy)
\]
with the convention that
\begin{equation}
\label{eq4.2} f(xe^y)=f(x_1e^{y_1},\dots, x_de^{y_d}),\quad
x=(x_1,\dots,x_d),\,y=(y_1,\dots,y_d)\in\BR^d.
\end{equation}

We have mentioned in the introduction that, in the present paper, we
reduce the problem of regularity of the value function $V$ to the
problem of regularity of the solution of the Cauchy problem
\begin{equation}
\label{eq3.9}
\partial_sv+L v=rv-g,\qquad v(T)=\psi,
\end{equation}
where $g\in L^2(0,T;L^2_{\varrho})$ with some suitably chosen
weight $\varrho$. By a standard change of variables, the last
problem reduces to the problem of regularity of the solution of
the Cauchy problem
\begin{equation}
\label{eq3.10} \partial_s\tilde v+\tilde L\tilde v
+\sum^d_{i=1}(r-\delta_i-\frac12a_{ii})\partial_{x_i}\tilde
v=r\tilde v-\tilde g,\qquad \tilde v(T)=\tilde \psi
\end{equation}
with suitably defined $\tilde g$, $\tilde \psi$ and with operator
$\tilde L$ being the infinitesimal generator of the semigroup on
$C_0(\BR^d)$ induced by the L\'evy process $\xi$, i.e.,
\begin{align*}
\tilde Lf(x)&=\frac12\sum^d_{i,j=1}a_{ij}\partial^2_{x_ix_j}f(x)
+\sum^d_{i=1}\gamma_i\partial_{x_i}f(x)\\
&\quad+\int_{\BR^d}\Big(f(x+y)-f(x)
-\sum^d_{i=1}y_i\fch_{\{|y|\le1\}}\partial_{x_i}f(x)\Big)\,\nu(dy)
\end{align*}
for $f\in C^2_0(\BR^d)$. The diffusion part of $\tilde L$ is a
uniformly elliptic operator, so to prove the regularity of $\tilde v$,
one can apply the methods of  the theory of parabolic equations
involving  integro-differential operators developed in Bensoussan and Lions (1982). It is worth pointing out, however, that the results of Bensoussan and Lions (1982) do not apply directly to our problem (in fact, they  provide
existence results under too-strong assumptions on $\psi,g$)
Therefore, in this section, we carefully investigate  problem
(\ref{eq3.9}). In our study, special emphasis is placed on the minimal
regularity assumptions on $\psi$ and the integrability assumptions on
the L\'evy measure $\nu$. At the end of this section, we provide a
stochastic representation of the solution of (\ref{eq3.9}).

\subsection{Variational solutions}
\label{sec4.1}

We assume that $\psi$ satisfies (\ref{eq3.1}). In what follows
\[
\rho(x)=e^{-\beta|x|}, \qquad\varrho(x)=e^{-\beta|\ln x|}
\cdot\frac{1}{|x_1\cdot\ldots\cdot x_d|^{1/2}}\,,\quad \quad x\in
D,
\]
where $\beta\ge0$ is  some constant and
\begin{equation}
\label{eq4.7} \ln x=(\ln(-1)^{i_1}x_1,\dots,\ln(-1)^{i_d}x_d))
\end{equation}
for $x\in D_{\iota}$ with $\iota=(i_1,\dots,i_d)$. In what follows
we will use some Sobolev spaces with weight $\varrho$ or $\rho$.
Our choice of the weights $\varrho,\rho$ will be justified in
Remark \ref{rem4.2}.

Let $\partial_t$, $\partial_{x_i}$, $i=1,\dots,d$, denote partial
derivatives in the distribution sense, and let
\[
L^2_{\varrho}=L^2(D;\varrho^2\,dx),\quad H^1_{\varrho}=\{u\in
L^2_{\varrho}:x_i\partial_{x_i}u\in L^2_{\varrho}\,,i=1,\dots,d\},
\]
\[
\WW^{0,1}_{\varrho}=\{u\in L^2(0,T;H^1_{\varrho}):\partial_tu\in
L^2(0,T;H^{-1}_{\varrho})\},
\]
where $H^{-1}_{\varrho}$ denotes the dual space of
$H^1_{\varrho}$. For $\varphi,\psi\in C^2_c(\BR^d)$ we set
\begin{equation}
\label{eq2.8} B_{\varrho}(\varphi,\psi)
=B^{BS}_{\varrho}(\varphi,\psi)+B^I_{\varrho}(\varphi,\psi),
\end{equation}
where
\[
B^{BS}_{\varrho}(\varphi,\psi)=-\frac12\sum^d_{i,j=1}a_{ij}
(\partial_{x_i}\varphi,\partial_{x_j}(x_ix_j\psi\varrho^2))_2
+\sum^d_{i=1}((r-\delta_i)x_i\partial_{x_i}\varphi,\psi\varrho^2)_2
\]
and
\[
B^{I}_{\varrho}(\varphi,\psi)=\int_{\BR^d}\Big(\int_{\BR^d}
\big(\varphi(xe^y)-\varphi(x)-\sum^d_{i=1}x_i(e^{y_i}-1)
\partial_{x_i}\varphi(x)\big)\,\nu(dy)\Big)\psi(x)\varrho^2(x)\,dx.
\]
In the above definitions, $(\cdot,\cdot)_2$ denotes the usual inner
product in $L^2(\BR^d;dx)$ and we use our convention
(\ref{eq4.2}). We will prove in Proposition \ref{prop4.3} that, if
$\beta\ge0$  and
\begin{equation}
\label{eq4.24} \int_{\{|y|>1\}}|y|e^{\beta|y|}\,\nu(dy)<\infty,
\end{equation}
then
\begin{equation}
\label{eq4.19} |B_{\varrho}(\varphi,\psi)|\le
c\|\varphi\|_{H^1_{\varrho}}\|\psi\|_{H^1_{\varrho}}
\end{equation}
for some $c>0$. Therefore, under (\ref{eq4.24}), the form
$B_{\varrho}$ can be extended to a bilinear form on
$H^1_{\varrho}\times H^1_{\varrho}$, which we still denote by
$B_{\varrho}$. Let us also observe that, for $\varphi\in
C^2_c(\BR^d)$, $\psi\in H^1_{\varrho}$, we have
\[
B_{\varrho}(\varphi,\psi)=(L\varphi,\psi)_{L^2_{\varrho}}
=(L\varphi,\psi\varrho^2)_2,
\]
where $L$ is defined by (\ref{eq4.1}).

Denote by $C([0,T];L^2_{\varrho})$ the space of all continuous
functions on $[0,T]$ with values in $L^2_{\varrho}$ equipped with
the norm $\|u\|_C=\sup_{0\le t\le T}\|u(t)\|_{L^2_{\varrho}}$. It
is known (see, e.g., Zhikov et al. (1981, theorem 2)) that there is a
continuous embedding of $\WW^{0,1}_{\varrho}$ into
$C([0,T];L^2_{\varrho})$. In particular, for every
$v\in\WW^{0,1}_{\varrho}$, one can find $w\in
C([0,T];L^2_{\varrho})$ such that $v(t)=w(t)$ for a.e.
$t\in[0,T]$. In what follows, we adopt the convention that any
element of $\WW^{0,1}_{\varrho}$ is already in
$C([0,T];L^2_{\varrho})$. With this convention, $v(T)$ is well
defined for $v\in\WW^{0,1}_{\varrho}$.

\begin{definition}
Let $\psi\in L^2_{\varrho}$, $g\in L^2(0,T;L^2_{\varrho})$ for
some $\beta\ge0$. We call  $v\in \WW^{0,1}_{\varrho}$ a
variational solution of the Cauchy problem (\ref{eq3.9}) if
$v(T)=\psi$  and for every $\eta\in C^{\infty}_c(Q_T)$,
\[
\int^T_0\langle\partial_tv(t),\eta(t)\rangle\,dt +\int^T_0
B_{\varrho}(v(t),\eta(t))\,dt=r\int_{Q_T}
v\eta\varrho^2\,dt\,dx-\int_{Q_{T}}\eta g\varrho^{2}\,dt\,dx,
\]
where  $\langle\cdot,\cdot\rangle$ denotes the duality pairing
between $H^{-1}_{\varrho}$ and $H^1_{\varrho}$.
\end{definition}

Now, set
\[
\tilde H^1_{\rho}=\{u\in L^2_{\rho}:\partial_{x_i}u\in
L^2_{\rho}\,,i=1,\dots,d\},
\]
\[
\tilde\WW^{0,1}_{\rho}=\{u\in L^2(0,T;L^2_{\rho}):
\partial_tu\in L^2(0,T;\tilde H^{-1}_{\rho})\},
\]
where $\tilde H^{-1}_{\rho}$ denotes the dual space of $\tilde
H^1_{\rho}$, and for $\varphi,\psi\in C^2_c(\BR^d)$ set
\[
\tilde B_{\rho}(\varphi,\psi)=\tilde B^{BS}_{\rho}(\varphi,\psi)
+\tilde B^I_{\rho}(\varphi,\psi),
\]
where
\begin{align}
\label{eq4.13} \tilde B^{BS}_{\rho}(\varphi,\psi)
&=-\frac12\sum^d_{i,j=1}a_{ij}
(\partial_{x_i}\varphi,\partial_{x_j}(\psi\rho^2))_2
+\sum^d_{i=1}((r-\delta_i+\gamma_i-\frac12a_{ii})
\partial_{x_i}\varphi,\psi\rho^2)_2\nonumber \\
&:=\tilde B^{BS,1}_{\rho}(\varphi,\psi)+\tilde
B^{BS,2}_{\rho}(\varphi,\psi)
\end{align}
and
\[
\tilde B^I_{\rho}(\varphi,\psi)=\int_{\BR^d}\Big(\int_{\BR^d}
\big(\varphi(x+y)-\varphi(x)-\sum^d_{i=1}y_i\fch_{\{|y|\le1\}}
\partial_{x_i}\varphi(x)\big)\,\nu(dy)\Big)\psi(x)\,\rho^2(x)\,dx.
\]
We will see in Proposition \ref{prop4.1} that $\tilde B_{\rho}$
can be extended to a bilinear form on $\tilde
H^1_{\rho}\times\tilde H^1_{\rho}$, which we still denote by
$\tilde B_{\rho}$.

Consider the space $C([0,T];L^2_{\rho})$ defined as
$C([0,T];L^2_{\varrho})$  but with $L^2_{\varrho}$ replaced by
$L^2_{\rho}$. Because the embedding of $\tilde\WW^{0,1}_{\rho}$ into
$C([0,T];L^2_{\rho})$ is continuous, as before, we may and will
assume that any element of $\tilde\WW^{0,1}_{\rho}$ is already in
$C([0,T];L^2_{\rho})$.

\begin{definition}
Let $\tilde\psi\in L^2_{\rho}$, $\tilde g\in L^2(0,T;L^2_{\rho})$
for some $\beta\ge0$. We call $\tilde v\in \tilde\WW^{0,1}_{\rho}$
a variational solution of the Cauchy problem (\ref{eq3.10}) if
$\tilde v(T)=\tilde\psi$,  and for every $\eta\in
C^{\infty}_c(Q_T)$,
\[
\int^T_0\langle\partial_t\tilde v(t),\eta(t)\rangle\,dt
+\int^T_0\tilde B_{\rho}(\tilde v(t),\eta(t))\,dt=r\int_{Q_T}
\tilde v\eta\rho^2\,dt\,dx-\int_{Q_{T}}\eta\tilde
g\rho^{2}\,dt\,dx,
\]
where $\langle\cdot,\cdot\rangle$ denotes the duality pairing
between $\tilde H^{-1}_{\rho}$ and $\tilde H^1_{\rho}$.
\end{definition}

\begin{proposition}
\label{prop4.1} Assume that $\tilde\psi\in L^2_{\rho}$, $\tilde
g\in L^2(0,T;L^2_{\rho})$  and \mbox{\rm(\ref{eq4.24})} is
satisfied for some $\beta\ge0$.  Then, there exists a unique
variational solution $\tilde v\in\tilde\WW^{0,1}_{\rho}$ of
\mbox{\rm(\ref{eq3.10})}. Moreover, there is $C>0$ such that
\begin{equation}
\label{eq4.5} \|\tilde v\|_{L^2(0,T;\tilde H^1_{\rho})}
+\|\partial_t\tilde v\|_{L^2(0,T;\tilde H^{-1}_{\rho})} \le
C(\|\tilde\psi\|_{L^2_{\rho}}+\|\tilde g\|_{L^2(0,T;L^2_{\rho})}).
\end{equation}
\end{proposition}
\begin{proof}
By making a standard change of variables, we may and will assume
that $r=0$. If we prove that
\begin{equation}
\label{eq4.8} |\tilde B_{\rho}(\varphi,\psi)|\le
c\|\varphi\|_{\tilde H^1_{\rho}}\cdot\|\psi\|_{\tilde
H^1_{\rho}}\,,\qquad \tilde B_{\rho}(\varphi,\varphi)\ge
a\|\varphi\|^2_{\tilde H^1_{\rho}}-b\|\varphi\|^2_{L^2_{\rho}}
\end{equation}
for some strictly positive  constant $a$ and positive $b,c$, then
the existence of a unique variational solution of (\ref{eq3.10})
and (\ref{eq4.5}) follows from Theorem 4.1 and Remark 4.3 in
Chapter 3 of Lions and Magenes (1968). The proof of (\ref{eq4.8}) in the case that $d=1$ is given in Mateche et al. (2004, appendix). Because the proof in the case that $d>1$
proceeds as in the case that $d=1$, with some modifications, here we
only sketch it. We provide, however, a detailed proof of estimates
for the nonlocal part of $\tilde B_{\rho}$ because it shows why
we adopt assumption (\ref{eq4.24}). As $C^2_c(\BR^d)$ is dense
in $L^2_\rho$ and in $\tilde H^1_{\rho}$, in the proof of
(\ref{eq4.8}), we may assume that $\varphi,\psi\in C^2_c(\BR^d)$.
We have
\begin{equation}
\label{eq4.09} (\varphi(x+y)-\varphi(x))\fch_{\{|y|>1\}}
=\sum^d_{i=1}\int^1_0y_i\fch_{\{|y|>1\}}\partial_{x_i}\varphi(x+\theta
y)\,d\theta
\end{equation}
and
\begin{align}
\label{eq4.9} &(\varphi(x+y)-\varphi(x)
-\sum^d_{i=1}y_i\partial_{x_i}\varphi(x))\fch_{\{|y|\le1\}}\\
&\qquad=\sum^d_{i=1}\int^1_0y_i \fch_{\{|y|\le1\}}
(\partial_{x_i}\varphi(x+\theta y)-\partial_{x_i}\varphi(x))\,d\theta
\nonumber\\
&\qquad= \sum^d_{i,j=1}\int^1_0\Big(\int^{\theta}_0y_iy_j
\fch_{\{|y|\le1\}}\partial^2_{x_ix_j}
\varphi(x+\theta'y)\,d\theta'\Big)\,d\theta.\nonumber
\end{align}
Hence,
\begin{align*}
\tilde B^I_{\rho}(\varphi,\psi) &=\int_{\BR^d}\Big(\int_{\BR^d}
\Big(\int^1_0\sum^d_{i=1}y_i\fch_{\{|y|>1\}}
\partial_{x_i}\varphi(x+\theta y) \,d\theta\Big)
\nu(dy)\Big)\psi(x)\rho^2(x)\,dx\\
&\quad+\int_{\BR^d}\Big(\int_{\BR^d}
\Big(\int^1_0\Big(\int^{\theta}_0\sum^d_{i,j=1}y_iy_j
\fch_{\{|y|\le1\}}\partial^2_{x_ix_j}
\varphi(x+\theta'y)\,d\theta'\Big)\,d\theta\Big)
\nu(dy)\Big) \\
&\qquad\qquad\qquad\qquad\qquad\qquad\qquad
\times\psi(x)\rho^2(x)\,dx=:I_1+I_2.
\end{align*}
As $\rho(x)/\rho(x+\theta y)\le e^{\beta|y|}$ and
\begin{align*}
&\int_{\BR^d}\sum^d_{i=1}\partial_{x_i} \varphi(x+\theta y)
\psi(x)\rho^2(x)\,dx\\
&\quad=\int_{\BR^d}\sum^d_{i=1}\partial_{x_i}\varphi(x+\theta y)
\rho(x+\theta y)\frac{\rho(x)}{\rho(x+\theta y)}\psi(x)\rho(x)
\,dx,
\end{align*}
applying  Fubini's theorem, we obtain
\begin{equation}
\label{eq4.10} |I_1|\le
c_1\int_{\BR^d}|y|e^{\beta|y|}\fch_{\{|y|>1\}}\,\nu(dy)\cdot
\|\partial_x\varphi\|_{L^2_{\rho}} \|\psi\|_{L^2_{\rho}},
\end{equation}
where $\partial_x=(\partial_{x_1},\dots,\partial_{x_d})$. To
estimate $I_2$, we first observe that
\begin{align*}
&\int_{\BR^d}\partial^2_{x_ix_j}\varphi(x+\theta'y)
\psi(x)\rho^2(x)\,dx\\
&\qquad=-\int_{\BR^d} \partial_{x_i}\varphi(x+\theta'y)
\{\partial_{x_j}\psi(x)
-2\beta\frac{x_j}{|x|}\psi(x)\}\rho^2(x)\,dx\\
&\qquad=-\int_{\BR^d}\partial_{x_i}\varphi(x+\theta'y)
\rho(x+\theta'y) \frac{\rho(x)}{\rho(x+\theta' y)}
\{\partial_{x_j}\psi(x)
-2\beta\frac{x_j}{|x|}\psi(x)\}\rho(x)\,dx.
\end{align*}
For $0<\delta\le1$, let $I_2^{\delta}$ denote the integral defined
as $I_2$ but with $\fch_{\{|y|\le1\}}$ replaced by
$\fch_{\{|y|\le\delta\}}$. As $\rho(x)/\rho(x+\theta y)\le
e^{\beta}$ if $|y|\le\delta\le 1$, it follows from the above
estimate and Fubini's theorem that
\begin{equation}
\label{eq4.11} |I^{\delta}_2|\le
c_2\int_{\BR^d}|y|^2\fch_{\{|y|\le\delta\}}\,\nu(dy)\cdot
\|\varphi\|_{\tilde H^1_{\rho}}\|\psi\|_{\tilde H^1_{\rho}}\,.
\end{equation}
By the  second equation in (\ref{eq4.9}) with $\fch_{\{|y|\le1\}}$
replaced by $\fch_{\{\delta<|y|\le1\}}$, we have
\begin{align*}
&\int_{\BR^d}\Big(\int^1_0\Big(\int^{\theta}_0\sum^d_{i,j=1}y_iy_j
\fch_{\{\delta<|y|\le1\}}\partial^2_{x_ix_j}
\varphi(x+\theta'y)\,d\theta'\Big)\,d\theta\Big)\psi(x)\rho^2(x)\,dx\\
&\qquad=\int_{\BR^d}\Big(\int^1_0\sum^d_{i=1}y_i
\fch_{\{\delta<|y|\le1\}} (\partial_{x_i}\varphi(x+\theta
y)-\partial_{x_i}\varphi(x))\,d\theta\Big)\psi(x)\rho^2(x)\,dx\\
&\qquad=\int_{\BR^d}\Big(\int^1_0\sum^d_{i=1}y_i
\fch_{\{\delta<|y|\le1\}} (\partial_{x_i}\varphi(x+\theta y)
\rho(x+\theta y)\frac{\rho(x)}{\rho(x+\theta y)}
\,d\theta\Big)\psi(x)\rho(x)\,dx\\
&\qquad\quad-\int_{\BR^d}\Big(\int^1_0\sum^d_{i=1}y_i
\fch_{\{\delta<|y|\le1\}}\partial_{x_i}\varphi(x))
\,d\theta\Big)\rho(x)\psi(x)\rho(x)\,dx.
\end{align*}
Hence
\begin{equation}
\label{eq4.12} |I_2-I^{\delta}_2|\le c_3
\int_{\BR^d}|y|\fch_{\{\delta<|y|\le1\}}\,\nu(dy)\cdot
\|\partial_x\varphi\|_{L^2_{\rho}}\|\psi\|_{L^2_{\rho}}.
\end{equation}
As $\nu$ is a L\'evy measure,
$\lim_{\delta\rightarrow0^+}I^{\delta}_2=0$. From this and
(\ref{eq4.11}), (\ref{eq4.12}), it follows that, for every
$\varepsilon\in(0,1)$, there exists $C_{\varepsilon}\ge0$ such that
$|I_2|\le \|\varphi\|_{\tilde
H^1_{\rho}}(\varepsilon\|\psi\|_{\tilde
H^1_{\rho}}+C_{\varepsilon}\|\psi\|_{L^2_{\rho}})$. By this and
(\ref{eq4.10}),
\begin{equation}
\label{eq4.14} |\tilde B^{I}_{\rho}(\varphi,\psi)|\le
c_4\|\partial_x\varphi\|_{L^2_{\rho}} \|\psi\|_{L^2_{\rho}} +
\|\varphi\|_{\tilde H^1_{\rho}}(\varepsilon\|\psi\|_{\tilde
H^1_{\rho}}+C_{\varepsilon}\|\psi\|_{L^2_{\rho}}).
\end{equation}
One can check that
\begin{equation}
\label{eq4.15} |\tilde B^{BS,1}_{\rho}(\varphi,\psi)|\le
c_5\|\varphi\|_{\tilde H^1_{\rho}}\|\psi\|_{\tilde H^1_{\rho}}\,,
\qquad \tilde B^{BS,1}_{\rho}(\varphi,\varphi)\ge
a_1\|\partial_x\varphi\|^2_{L^2_{\rho}}
-b_1\|\varphi\|^2_{L^2_{\rho}}
\end{equation}
for some strictly positive constants $c_5,a_1,b_1$ (in the proof
of the second inequality, we use (\ref{eq2.02})). Moreover,
\begin{equation}
\label{eq4.16} |\tilde B^{BS,2}_{\rho}(\varphi,\psi)|\le
c_6\|\partial_x\varphi\|_{L^2_{\rho}} \|\psi\|_{L^2_{\rho}}.
\end{equation}
From (\ref{eq4.14})--(\ref{eq4.16}),  we deduce (\ref{eq4.8}) by
standard calculations.  This completes the proof of the
proposition.
\end{proof}

Given $\psi:D\rightarrow\BR$, $g:[0,T]\times D\rightarrow\BR$, let
us set
\begin{equation}
\label{eq4.22} \tilde \psi(x)
=\psi((-1)^{i_1}e^{x_1},\dots,(-1)^{i_d}e^{x_d}),\quad \tilde
g(t,x) =g(t,(-1)^{i_1}e^{x_1},\dots,(-1)^{i_d}e^{x_d})
\end{equation}
if  $x\in D_{\iota}$ with $\iota=(i_1,\dots,i_d)$. Observe that,
with this notation,
\[
\tilde \psi(\ln x)=\psi(t,x),\quad\tilde g(t,\ln x)=g(t,x),\quad
x\in D,
\]
where $\ln x$ is defined by (\ref{eq4.7}).

\begin{remark}
\label{rem4.2} (i) $\psi\in L^2_{\varrho}$ if and only if $\tilde
\psi\in L^2_{\rho}$ and $ g\in L^2(0,T;L^2_{\varrho})$ if and only
if $\tilde g\in L^2(0,T;L^2_{\rho})$, because, by the change of
variables formula, for any measurable $f:[0,T]\times
D\rightarrow[0,\infty)$, we have
\begin{align}
\label{eq4.17} \int_{\BR^d}f(t,x)\varrho^2(x)\,dx =\int_{P} \tilde
f(t,\ln x) \varrho^2(x)\,dx =\int_{\BR^d}\tilde
f(t,x)\rho^2(x)\,dx.
\end{align}
(ii) If a measurable $\psi:D\rightarrow\BR_+$ satisfies
(\ref{eq3.1}) and $\beta>p$, then $\psi\in L^2_{\varrho}$. Indeed,
we have
\[
\int_{\BR^d}|\tilde\psi(x)|^2\rho^2(x)\,dx\le c
\int_{\BR^d}(1+e^{p|x|})^2 e^{-2\beta|x|}\,dx<\infty
\]
for some $c$ depending on $d,p$. Hence, $\tilde\psi\in L^2_{\rho}$,
and consequently, $\psi\in L^2_{\varrho}$.
\end{remark}

\begin{proposition}
\label{prop4.3} Let $\beta>p$. Assume that $\psi$  satisfies
\mbox{\rm(\ref{eq3.1})},  $\nu$ satisfies \mbox{\rm(\ref{eq2.1}),
(\ref{eq4.24})} and  $g\in L^2(0,T;L^2_{\varrho})$. Then, there
exists a unique variational solution $v\in\WW^{0,1}_{\varrho}$ of
the Cauchy problem \mbox{\rm(\ref{eq3.9})}.
\end{proposition}
\begin{proof}
We first show (\ref{eq4.19}). Let $\varphi,\psi\in C^2_c(\BR^d)$,
and let $\tilde\varphi,\tilde\psi$ be defined by (\ref{eq4.22}).
Then
\begin{align*}
B^I_{\varrho}(\varphi,\psi)&=\int_{P}\Big(\int_{\BR^d}
\big(\tilde\varphi(y+\ln x)-\tilde\varphi(\ln x)
-\sum^d_{i=1}y_i\fch_{\{|y_i|<1\}}
\partial_{x_i}\tilde \varphi(\ln x)\big)\,\nu(dy)\Big)\\
&\qquad\qquad\qquad\qquad\qquad\qquad\qquad\qquad\qquad\qquad
\quad\times\tilde \psi(\ln x)\,\varrho^2(x)\,dx\\
&\quad-\int_{P}\Big(\int_{\BR^d} \big(
\sum^d_{i=1}(e^{y_i}-1-y_i\fch_{\{|y_i|<1\}})
\partial_{x_i}\tilde \varphi(\ln x)\big)\,\nu(dy)\Big)\tilde
\psi(\ln x)\,\varrho^2(x)\,dx\\
&=:I_1+I_2.
\end{align*}
Changing the variables $x_k\mapsto (-1)^{i_k}e^{z_k}$, we obtain
\[
I_1=\int_{\BR^d}\Big(\int_{\BR^d} \big(\tilde\varphi(y+z)
-\tilde\varphi(z) -\sum^d_{i=1}y_i\fch_{\{|y_i|<1\}}
\partial_{z_i}\tilde \varphi(z)\big)\,\nu(dy)\Big)
\tilde \psi(z)\,\rho^2(z)\,dz.
\]
Therefore, the proof of Proposition \ref{prop4.1} shows that there
is $c_2>0$ such that $|I_1|\le c\|\tilde\varphi\|_{\tilde
H^1_{\rho}}\|\tilde\psi\|_{\tilde H^1_{\rho}} $ if (\ref{eq4.24})
is satisfied. Hence,
\begin{equation}
\label{eq4.21} |I_1|\le c_2
\|\varphi\|_{H^1_{\varrho}}\|\psi\|_{H^1_{\varrho}}
\end{equation}
because by (\ref{eq4.17}), $\|\tilde\varphi\|_{\tilde
H^1_{\rho}}=\|\varphi\|_{H^1_{\varrho}}$ for any $\varphi\in
C^1_c(\BR^d)$. From (\ref{eq2.2})  it follows easily that
(\ref{eq4.21}) (perhaps with different constant) holds for $I_1$
replaced by $I_2$. Using (\ref{eq4.17}), one can also check that
\[
B^{BS}_{\varrho}(\varphi,\psi)\le
c_1\|\varphi\|_{H^1_{\varrho}}\|\psi\|_{H^1_{\varrho}}
\]
for some $c_1\ge0$,
which with the estimates for $I_1$ and $I_2$  yield
(\ref{eq4.19}). Thus, the form $B_{\varrho}$ is well defined. Now,
let $\psi$ denote the function appearing in the formulation of the
proposition and let $\tilde\psi(x)$, $\tilde g(t,x)$ be defined by
(\ref{eq4.22}). By Remark \ref{rem4.2}, $\tilde\psi\in L^2_{\rho}$
and $\tilde g\in L^2(0,T;L^2_{\rho})$. Therefore, by Proposition
\ref{prop4.1}, there exists  a unique solution $\tilde
v\in\tilde\WW^{0,1}_{\rho}$ of (\ref{eq3.10}). Define
$v:[0,T]\times D\rightarrow\BR$ as
\[
v(t,x)=\tilde v(t,\ln x),\quad t\in[0,T],\,x\in D.
\]
From the fact that $\tilde v\in\tilde\WW^{0,1}_{\rho}$, equalities
(\ref{eq4.17}) with  $f$ replaced by $\tilde v$ and similar
equalities with $f$ replaced by $x_i\partial_{x_i}\tilde v$, it
follows that $v\in\WW^{0,1}_{\varrho}$. One can also check that, if
$\tilde v$ satisfies (\ref{eq3.10}), then $v$ satisfies
(\ref{eq3.9}) (in the calculations, we use (\ref{eq2.2})), which
completes the proof.
\end{proof}

\begin{remark} 
For every $\beta\ge0$ and $x\in D_{\iota}$ with
$\iota=(i_1,\dots,i_d)$, we have
\begin{align*}
\varrho^2(x)&\ge
e^{-2\beta(|\ln(-1)^{i_1}x_1|+\ldots+|\ln(-1)^{i_d}x_d|)}
\frac{1}{|x_1\cdot\ldots\cdot x_d|}\\
&=\prod_{k: |x_k|\ge1} \frac{1}{|x_k|^{2\beta+1}}\,\cdot\prod_{k:
0<|x_k|<1}|x_k|^{2\beta-1}.
\end{align*}
\end{remark}

\subsection{Improved regularity  and stochastic representation}

Set
\[
H^2_{\varrho}=\{u\in L^2_{\varrho}:x_i\partial_{x_i}u\in
L^2_{\varrho}\,,\,x_ix_j\partial^2_{x_ix_j}u\in
L^2_{\varrho}\,,\,i,j=1,\dots,d\},
\]
\[
W^{1,2}_{\varrho}=\{u\in L^2(0,T;H^2_{\varrho}):\partial_tu\in
L^2(0,T;L^2_{\varrho})\}.
\]
and
\[
\tilde H^2_{\rho}=\{u\in L^2_{\rho}:\partial_{x_i}u\in
L^2_{\rho}\,,\,\partial^2_{x_ix_j}u\in L^2_{\rho}\,,\,i,j=1,\dots
d\},
\]
\[
\tilde W^{1,2}_{\rho}=\{u\in L^2(0,T;\tilde
H^2_{\rho}):\partial_tu\in L^2(0,T;L^2_{\rho})\}.
\]
In the case that $\beta=0$ (i.e., $\rho\equiv1$), we will omit the subscript
 $\rho$ in the above notation.

For $\varphi\in C^2_c(\BR^d)$, set
\begin{align*}
\tilde L_I\varphi(x)&=\int_{\BR^d}(\varphi(x+y)-\varphi(x)-
\sum_{i=1}^dy_i\partial_{x_i}\varphi(x))\fch_{\{|y|\le1\}}\,\nu(dy)\\
&\quad+\int_{\BR^d}(\varphi(x+y)-\varphi(x))
\fch_{\{|y|>1\}}\,\nu(dy):=\tilde L^1_I+\tilde L^2_I.
\end{align*}
By Bensoussan and Lions (1982, lemma 3.1.3), for  $r>0$, there exist constants
$a(r)$, $b(r)$ such that $a(r)\rightarrow0$ as $r\rightarrow0$ and
\begin{equation}
\label{eq4.30} \|\tilde L^1_I\varphi\|_{L^2_{\rho}}\le
c(a(r)\|\varphi\|_{\tilde H^2_{\rho}}
+b(r)\|\varphi\|_{L^2_{\rho}}).
\end{equation}
(in fact, this can be shown by using (\ref{eq4.9}) and modifying
the argument from the proof of (\ref{eq4.14})). Let
$c(\nu)=\nu(\{y:|y|>1\})$. By (\ref{eq4.09}), we have
\begin{align*}
\|\tilde L^2_I\varphi\|^2_{L^2_{\rho}}
&=\int_{\BR^d}\Big|\int_{\BR^d}
\Big(\sum^d_{i=1}\int^1_0y_i\fch_{\{|y|>1\}}\partial_{x_i}
\varphi(x+\theta y)\,d\theta\Big)\nu(dy)\Big|^2\rho^2(x)\,dx\\
&\le c(\nu)\int_{\BR^d}\Big(\int_{\BR^d}
\Big|\sum^d_{i=1}\int^1_0y_i\fch_{\{|y|>1\}}\partial_{x_i}
\varphi(x+\theta y)\,d\theta\Big|^2\,\nu(dy)\Big)\rho^2(x)\,dx\\
&\le c(\nu)\int_{\BR^d} \Big(\int_{\BR^d}
\int^1_0|y|^2\fch_{\{|y|>1\}}|\partial_{x_i}
\varphi(x+\theta y)|^2\,d\theta\,\nu(dy)\Big)\rho^2(x)\,dx\\
&=c(\nu)\int_{\BR^d}\int_{\BR^d}\int^1_0|y|^2
\fch_{\{|y|>1\}}|\partial_{x_i}
\varphi(x+\theta y)|^2\rho^2(x+\theta y)
\frac{\rho^2(x)}{\rho^2(x+\theta y)} \,d\theta\,\nu(dy)\,dx\\
&\le c(\nu)\int_{\BR^d}|y|^2e^{2\beta|y|}\fch_{\{|y|>1\}}\,\nu(dy)
\int_{\BR^d}|\partial_x\varphi(x)|^2\rho^2(x)\,dx.
\end{align*}
As a consequence, if
\begin{equation}
\label{eq4.33} \int_{\{|y|>1\}}|y|^2e^{2\beta|y|}\,\nu(dy)<\infty,
\end{equation}
then
\begin{equation}
\label{eq4.31} \|\tilde L^2_I\varphi\|_{L^2_{\rho}}\le c
\|\partial_x\varphi\|_{L^2_{\rho}} \le
c(\frac{\varepsilon}{2}\|\varphi\|_{\tilde
H^2_{\rho}}+\frac{1}{2\varepsilon}\|\varphi\|_{L^2_{\rho}})
\end{equation}
for $\varepsilon>0$.
Thus, if (\ref{eq4.33}) is satisfied, then the operator $\tilde
L_I$ may be extended to an operator on $\tilde H^2_{\rho}$. This
extension  will  still be denoted by $\tilde L_I$.

\begin{lemma}
\label{lem4.5} If $\tilde\psi\in\tilde H^1_{\rho}$ and $\tilde
g\in L^2(0,T;L^2_{\rho})$ for some $\beta\ge0$, then there exists a
unique solution $u\in\tilde W^{1,2}_{\rho}$ of the Cauchy problem
\begin{equation}
\label{eq4.38}
\partial_tu-\tilde L_{BS}u=\tilde g,\quad u(0)=\tilde\psi.
\end{equation}
Moreover, there is $c(\rho)>0$ depending only on $\rho$ such that
\begin{equation}
\label{eq4.39} \|u\|_{\tilde W^{1,2}_{\rho}}\le
c(\rho)(\|\tilde\psi\|_{\tilde H^1_{\rho}}+\|\tilde
g\|_{L^2(0,T;L^2_{\rho})}).
\end{equation}
\end{lemma}
\begin{proof}
Choose $\psi_N\in\tilde H^1$, $g_N\in L^2(0,T;L^2)$ so that
$\psi_N\rightarrow\tilde\psi$ in $\tilde H^1_{\rho}$ and
$g_N\rightarrow\tilde g$ in $L^2(0,T;L^2_{\rho})$. By classical
results (see, e.g., Garroni and Menaldi (1992, theorem V.4.2), for each $N$, there exists a unique solution $u_N\in \tilde W^{1,2}$ of the problem
\begin{equation}
\label{eq4.35}
\partial_tu_N-\tilde L_{BS}u_N=g_N,\quad u_N(0)=\psi_N.
\end{equation}
By Proposition \ref{prop4.1} (with $\nu\equiv0$),
\begin{equation}
\label{eq4.36} \|u_N\|_{L^2(0,T;\tilde H^1_{\rho})}\le
c(\|\psi_N\|_{L^2_{\rho}} +\|g_N\|_{L^2(0,T;L^2_{\rho})}).
\end{equation}
We check  by direct calculation  that $u_N\cdot\rho\in\tilde
W^{1,2}$ is a solution of the problem
\[
(\partial_t-\tilde L_{BS})(u_N\cdot\rho)=g_N\cdot\rho+h_N,\quad
u_N\cdot\rho(0)=\psi_N\cdot\rho
\]
with some $h_N\in L^2(0,T;L^2_{\rho})$ such that
\begin{equation}
\label{eq4.37} \|h_N\|_{L^2(0,T;L^2_{\rho})}\le
c\|u_N\|_{L^2(0,T;L^2_{\rho})}\|\rho\|_{\tilde H^2_{\rho}}.
\end{equation}
By Garroni and Menaldi (1992, theorem V.4.2) and (\ref{eq4.36}), (\ref{eq4.37}),
\begin{align*}
\|u_N\cdot\rho\|_{\tilde W^{1,2}}&\le
c(\|\psi_N\cdot\rho\|_{\tilde
H^1}+\|g_N\cdot\rho+h_N\|_{L^2(0,T;L^2)})\\
&\le c_1(\rho)(\|\psi_N\|_{\tilde H^1_{\rho}}
+\|g_N\|_{L^2(0,T;L^2_{\rho})}+\|u_N\|_{\tilde W^{0,1}_{\rho}})\\
&\le c_2(\rho)(\|\psi_N\|_{\tilde H^1_{\rho}}
+\|g_N\|_{L^2(0,T;L^2_{\rho})})
\end{align*}
for some constants $c_1(\rho),c_2(\rho)$ depending only on $\rho$.
As
\[
\|\partial_tu_N\|_{L^2(0,T;L^2_{\rho})}
=\|(\partial_tu_N)\cdot\rho\|_{L^2(0,T;L^2)}
=\|\partial_t(u_N\cdot\rho)\|_{L^2(0,T;L^2)},
\]
it follows in particular that
\[
\|\partial_tu_N\|_{L^2(0,T;L^2_{\rho})} \le
c_2(\rho)(\|\psi_N\|_{\tilde H^1_{\rho}}
+\|g_N\|_{L^2(0,T;L^2_{\rho})}).
\]
From this, (\ref{eq4.35}), (\ref{eq4.36}) and (\ref{eq2.02}), we
deduce that
\begin{equation}
\label{eq4.40} \|u_N\|_{\tilde W^{1,2}_{\rho}} \le
c(\rho)(\|\psi_N\|_{\tilde H^1_{\rho}}
+\|g_N\|_{L^2(0,T;L^2_{\rho})}).
\end{equation}
By Proposition \ref{prop4.1}, there exists a unique variational
solution $u\in\tilde\WW^{0,1}_{\rho}$ of (\ref{eq4.38}), and by
(\ref{eq4.5}), $u_N\rightarrow u$ in $\tilde\WW^{0,1}_{\rho}$ as
$N\rightarrow\infty$. From this and (\ref{eq4.40}), we conclude
that $u\in\tilde W^{1,2}_{\rho}$ and (\ref{eq4.39}) is satisfied.
\end{proof}

\begin{proposition}
\label{prop4.5}  Assume that  $\tilde\psi\in\tilde H^1_{\rho}$,
$\tilde g\in L^2(0,T;L^2_{\rho})$ and \mbox{\rm(\ref{eq4.33})} is
satisfied for some $\beta\ge0$. Then, the variational solution
$\tilde v$ of \mbox{\rm(\ref{eq3.10})} belongs to $\tilde
W^{1,2}_{\rho}$.
\end{proposition}
\begin{proof}
We first show that there exists a unique solution $u\in\tilde
W^{1,2}_{\rho}$ of the Cauchy problem
\begin{equation} \label{eq4.32}
\partial_tu-\tilde Lu=-g,\quad u(0)=\tilde\psi,
\end{equation}
where $g(t,x)=\tilde g(T-t,x)$. To see this, we define $F:\tilde
W^{1,2}_{\rho}\rightarrow\tilde W^{1,2}_{\rho}$ by putting $F(w)$
to be a unique solution $u\in \tilde W^{1,2}_{\rho}$ of the Cauchy
problem
\[
\partial_tu-\tilde L_{BS}u=-g+\tilde L_Iw,\quad u(0)=\tilde\psi.
\]
By (\ref{eq4.30}) and (\ref{eq4.31}),  $\tilde L_Iw\in
L^2(0,T;L^2_{\rho})$, so by Lemma \ref{lem4.5} the mapping $F$ is
well defined. Let $u_1=F(w_1), u_2=F(w_2)$ for some
$w_1,w_2\in\tilde W^{1,2}_{\rho}$, and let $w=w_1-w_2,u=u_1-u_2$.
Then, $u$ is a solution of the problem
\[
\partial_su-\tilde L_{BS}u=\tilde L_Iw,\quad u(0)=0.
\]
Using  (\ref{eq4.39}) and then (\ref{eq4.30}), (\ref{eq4.31}), we
obtain
\begin{align*}
\|u\|_{\tilde W^{1,2}_{\rho}}&\le c\|\tilde
L_Iw\|_{L^2(0,T;L^2_{\rho})} \\
&\le c_1(a(r)+\frac{\varepsilon}{2})\int^T_0\|w(t)\|_{\tilde
H^2_{\rho}}\,dt
+(b(r)+\frac{1}{2\varepsilon})\int^T_0\|w(t)\|_{L^2_{\rho}}\,dt.
\end{align*}
As $w(0)=0$ a.e. and $w\in\tilde W^{1,2}_{\rho}$,
\[
\|w(t)\|_{L^2_{\rho}}\le
t^{1/2}(\int^t_0\|\partial_sw(s)\|^2_{L^2_{\rho}}\,ds)^{1/2} \le
t^{1/2}\|w\|_{\tilde W^{1,2}_{\rho}}\,.
\]
Hence
\[
\|u\|_{\tilde W^{1,2}_{\rho}}  \le
c_1(a(r)+\frac{\varepsilon}{2})T^{1/2} \|w\|_{\tilde
W^{1,2}_{\rho}} +(b(r)+\frac{1}{2\varepsilon})T^{3/2}\|w\|_{\tilde
W^{1,2}_{\rho}}\,.
\]
Therefore, choosing first $r,\varepsilon$ so that
$c_1(a(r)+\frac{\varepsilon}{2})T^{1/2}\le1/4$ and then choosing
$T_0\le T$ so that $(b(r)+\frac{1}{2\varepsilon})T_0^{3/2}\le1/4$,
we see that $F$ is a contraction for $T:=T_0$. Therefore, there
exists a unique solution $u$ of  (\ref{eq4.32}) with $T$ replaced
by $T_0$ and hence, by the standard argument, a unique solution $u$
of (\ref{eq4.32}). The function $v$ defined as $v(t,x)=u(T-t,x)$
is then a unique strong solution of (\ref{eq3.10}). Of course,
a strong solution is a variational solution, so $v=\tilde v$ by
uniqueness. This proves the proposition.
\end{proof}

As a corollary to Proposition \ref{prop4.5} we get improved
regularity of  the solution of (\ref{eq4.9}) under  assumption
(\ref{eq4.33}).

\begin{proposition}
\label{prop4.6} Let $\beta>p$. Assume that $\psi$ satisfies
\mbox{\rm(\ref{eq3.1})} and $\tilde\psi\in\tilde H^1_{\rho}$,
where $\tilde\psi$ is defined by \mbox{\rm(\ref{eq4.22})}.
Moreover, assume that $\nu$ satisfies \mbox{\rm(\ref{eq2.1}),
(\ref{eq4.33})} and $g\in L^2(0,T;L^2_{\varrho})$. Then, the
variational solution $v$ of \mbox{\rm(\ref{eq3.9})} belongs to
$W^{1,2}_{\varrho}$.
\end{proposition}
\begin{proof}
From the proof of Proposition \ref{prop4.3}, it follows that
\begin{equation}
\label{eq4.34} v(t,x)=\tilde v(t,\ln x),\quad t\in[0,T],x\in D,
\end{equation}
where $\tilde v\in\tilde W^{0,1}_{\rho}$ is a  variational
solution of (\ref{eq3.10}) with $\tilde\psi,\tilde g$ defined by
(\ref{eq4.22}). By Remark \ref{rem4.2}, $\tilde g\in
L^2(0,T;L^2_{\rho})$. Therefore, $\tilde v\in\tilde W^{1,2}_{\rho}$
by Proposition \ref{prop4.5}. Using (\ref{eq4.17}) and
(\ref{eq4.34}), one can check that $v\in W^{1,2}_{\varrho}$.
\end{proof}

\begin{remark}
\label{rem4.8} If $\psi:\BR^d\rightarrow\BR_+$ is Lipschitz
continuous, then $\psi$ satisfies (\ref{eq3.1}) and
$\tilde\psi\in\tilde H^1_{\rho}$.
\end{remark}

\begin{proposition}
\label{prop4.9} Let  $\beta>p$ and let $(a,\nu,\gamma)$ satisfy
\mbox{\rm(\ref{eq2.2}), (\ref{eq2.02}), (\ref{eq2.06}),
(\ref{eq4.24})}. Assume that $\psi:\BR^d\rightarrow\BR_+$
satisfies \mbox{\rm(\ref{eq3.1})}, that $\psi_{|D}$ is continuous, and
that $g\in L^2(0,T;L^2_{\varrho})$ is a nonnegative function such
that $E_{s,x}\int^T_sg(t,X_t)\,dt<\infty$ for every
$(s,x)\in[0,T]\times D$.
\begin{enumerate}
\item[\rm(i)]There exists a unique  function $v:[0,T]\times
D\rightarrow\BR$ such that
\[
E_{s,x}\int^T_s|v(t,X_t)|\,dt<\infty
\]
and
\begin{equation}
\label{eq3.8} v(s,x)=E_{s,x}\Big(\psi(X_T)+\int^T_s(-rv (t,X_t)
+g(t,X_t))\,dt\Big)
\end{equation}
for every $(s,x)\in[0,T]\times D$.

\item[\rm(ii)]$v\in\WW^{0,1}_{\varrho}$
and $v$ is a variational
solution of the Cauchy problem \mbox{\rm(\ref{eq3.9})}.

\item[\rm(iii)]
Assume additionally that $\tilde\psi\in\tilde H^1_{\rho}$, where
$\tilde\psi$ is defined by \mbox{\rm(\ref{eq4.22})}, and that
$\nu$ satisfies \mbox{\rm(\ref{eq4.33})}. Then, $v\in
W^{1,2}_{\varrho}$ and
\begin{equation}
\label{eq4.44}
\partial_sv+L v=rv-g\mbox{ a.e. in }Q_T.
\end{equation}
\end{enumerate}
\end{proposition}
\begin{proof}
For $(s,x)\in[0,T]\times D$ set
\begin{equation} \label{eq3.15}
w(s,x)=E_{s,x}\Big(e^{-rT}\psi(X_T)+\int^T_se^{-rt}g(t,X_t)\,dt\Big)
\end{equation}
and $Y_t=w(t,X_t)$, $t\in[s,T]$. Let  $M^{s,x}$ be a c\`adl\`ag
martingale on the probability space $(\Omega,(\FF^s_t),P_{s,x})$
defined as
\[
M^{s,x}_t=E_{s,x}\Big(e^{-rT}\psi(X_T) +\int^T_s
e^{-r\theta}g(\theta,X_{\theta})\,d\theta\big|\FF^s_t\Big)
-w(s,X_s),\quad t\in[s,T].
\]
Then
\begin{align*}
&e^{-rT}\psi(X_T)+\int^T_te^{-r\theta}g(\theta,X_{\theta})\,d\theta
-\int^T_tdM^{s,x}_{\theta} \\
&\qquad=E_{s,x}\Big(e^{-rT}\psi(X_T) +\int^T_te^{-r\theta}
g(\theta,X_{\theta})\,d\theta\big|\FF^s_t\Big),
\end{align*}
By the Markov property and (\ref{eq3.15}), the right-hand side of
the above equality is equal to $w(t,X_t)$. Hence, the pair
$(Y,M^{s,x})$ is a solution of the BSDE
\[
Y_t=e^{-rT}\psi(X_T)+\int^T_te^{-r\theta}g(\theta,X_{\theta})\,d\theta
-\int^T_tdM^{s,x}_{\theta},\quad t\in[s,T]
\]
on  $(\Omega,(\FF^s_t),P_{s,x})$. Integrating by parts, we obtain
\[
e^{rs}Y_s=\psi(X_T)+\int^T_s(-re^{rt}Y_{t}+g(t,X_t))\,dt
+\int^T_se^{rt}\,dM^{s,x}_t.
\]
Therefore, $v$ defined as $v(s,x)=e^{rs}w(s,x)$,
$(s,x)\in[0,T]\times D$, satisfies (\ref{eq3.8}). Because in much
the same way,  one can show that, if $v$ satisfies (\ref{eq3.8}),
then $w$ defined as $w(s,x)=e^{-rs}v(s,x)$ is given by
(\ref{eq3.15}), $v$ satisfying (\ref{eq3.8}) is unique. This
proves (i).

It is easily seen that $v$ is a variational solution of
(\ref{eq3.9}) in $\WW^{0,1}_{\varrho}$ if and only if $w$ is a
variational solution of the problem
\begin{equation}
\label{eq3.16}
\partial_sw+L w=-e^{-rs}g, \qquad w(T,\cdot)=e^{-rT}\psi
\end{equation}
in $\WW^{0,1}_{\varrho}$. Therefore, it suffices to show that $w$
defined by  (\ref{eq3.15}) is a variational solution of
(\ref{eq3.16}) in the space $\WW^{0,1}_{\varrho}$. Set
\[
\eta_t=\ln X_t=(\ln(-1)^{i_1}X^1_t,\dots,\ln(-1)^{i_d}X^d_t),\quad
t\in[s,T]
\]
and define  $\tilde w,\tilde \psi,\tilde g$ by (\ref{eq4.22}),
i.e., $ \tilde w(t,x)=w(t,\tilde x)$, $\tilde\psi(x)=\psi(\tilde
x)$, $ \tilde g(t,x)=g(t,\tilde x)$, where $\tilde
x=((-1)^{i_1}e^{x_1},\dots,(-1)^{i_d}e^{x_d})$ for $x\in
D_{\iota}$ with $\iota=(i_1,\dots,i_d)$. Observe that for every
$x\in D$ the process $\eta$ is under $P_{s,\tilde x}$  a L\'evy
process with the characteristic triplet $(a,\nu,\gamma+r-\delta)$
starting at time $s$ from $x$. By (\ref{eq3.15}), for every
$(s,x)\in[0,T]\times D$ we have
\begin{equation}
\label{eq3.11} \tilde w(s,x) =E_{s,\tilde x }
\Big(e^{-rT}\tilde\psi(\eta_T)+\int^T_se^{-rt}\tilde
g(t,\eta_t)\,dt\Big).
\end{equation}
By Remark \ref{rem4.2},  $\tilde g\in L^2(0,T;L^2_{\rho})$.
Suppose for the moment that $\tilde\psi$ is bounded. Let
$\{\psi_n\}\subset C^2_c(\BR^d)$ be a sequence such that
$\sup_{n\ge1}\|\psi\|_{\infty}<\infty$ and
$\psi_n\rightarrow\tilde\psi$ in $L^2_{\rho}$, and let $g_n=\tilde
g\wedge n$. Then, by Bensoussan and Lions (1982, theorem 3.3.3), for each $n\in\BN$, there exists a unique variational solution
$w_n\in\tilde\WW^{0,1}_{\rho} $ (in fact $w_n\in\tilde
W^{1,2}_{\rho}$) of the problem
\begin{equation}
\label{eq3.12}
\partial_s w_n+\tilde Lw_n
+\sum^d_{i=1}(r-\delta_i-\frac12a_{ii})\partial_{x_i}w_n=-e^{-rs}
g_n,\qquad  w_n(T,\cdot)=e^{-rT}\psi_n.
\end{equation}
By Bensoussan and Lions (1982, theorem 3.8.1), $w_n$ has the representation
\begin{equation}
\label{eq3.13} w_n(s,x) =E_{s,\tilde x}\Big(e^{-rT}\psi_n(\eta_T)
+\int^T_se^{-rt}g_n(t,\eta_t)\,dt\Big).
\end{equation}
By a priori estimate (\ref{eq4.5}), $\{w_n\}$ converges in
$\tilde\WW^{0,1}_{\rho}$ to the unique variational solution $\hat
w\in\tilde\WW^{0,1}_{\rho}$ of  problem
\begin{equation}
\label{eq3.17}
\partial_s\hat w+\tilde L\hat w
+\sum^d_{i=1}(r-\delta_i-\frac12a_{ii})\partial_{x_i} \hat
w=-e^{-rt}\tilde g,\qquad \hat w(T,\cdot)=e^{-rT}\tilde \psi.
\end{equation}
On the other hand, $E_{s,\tilde x}e^{-rT}\psi_n(\eta_T)\rightarrow
E_{s,\tilde x}e^{-rT}\tilde\psi(\eta_T)$ by the dominated
convergence, and $E_{s,\tilde x}
\int^T_se^{-rt}g_n(t,\eta_t)\,dt\rightarrow E_{s,\tilde x}
\int^T_se^{-rt}\tilde g(t,\eta_t)\,dt$ by the monotone
convergence, so the right-hand side of (\ref{eq3.13}) converges to
the right-hand side of (\ref{eq3.11}), that is, $w_n\rightarrow
\tilde w$ pointwise. It follows that  if $\tilde \psi$ is bounded
then $\tilde w$ is a version of the solution $\hat w$ of
(\ref{eq3.17}). Consider now the general case.  For $k\in\BN$, set
$\tilde\psi_k=\tilde\psi\wedge k$. By what has already been
proved, $\tilde w_k$ defined by
\begin{equation}
\label{eq4.42} \tilde w_k(s,x) =E_{s,\tilde x }
\Big(e^{-rT}\tilde\psi_k(\eta_T)+\int^T_se^{-rt}\tilde
g(t,\eta_t)\,dt\Big)
\end{equation}
is a version of the solution $\hat w_k$  of (\ref{eq3.17}) with
$\tilde\psi$ replaced by $\tilde\psi_k$. By (\ref{eq4.5}),
$\{\tilde w_k\}$ converges in $\tilde\WW^{0,1}_{\rho}$ to the
unique variational solution $\hat w\in\tilde\WW^{0,1}_{\rho}$ of
(\ref{eq3.17}). By monotone convergence, the right-hand side of
(\ref{eq4.42}) converges to  the right-hand side of
(\ref{eq3.11}). Thus, for $\psi,g$ satisfying the assumptions of
the proposition, the function  $\tilde w$ is a version of the
solution $\hat w$ of (\ref{eq3.17}), and hence $w$ defined by
(\ref{eq3.15}) is a variational solution of (\ref{eq3.16}) (see
the end of the proof of Proposition \ref{prop4.3}). This completes
the proof of part (ii).

By Proposition \ref{prop4.6}, $v\in W^{1,2}_{\varrho}$. From this
and part (ii), it follows that, for every $\eta\in
C^{\infty}_c((0,T)\times\BR^d)$,
\[
\int^T_0\big(\partial_tv(t)+Lv(t) -rv(t)+
g(t),\eta(t)\varrho^2\big)_2\,dt=0,
\]
which implies (\ref{eq4.44}).
\end{proof}

\section{Obstacle problem and reflected BSDEs} \label{sec5}

In this section, we consider the obstacle problem
\begin{equation}
\label{eq4.3} \min\{-\partial_s u-L u+ru,u-\psi\}=0,\quad
u(T)=\psi
\end{equation}
associated with the  optimal stopping problem (\ref{eq2.09})
(and hence with the reflected BSDE (\ref{eq3.06})). For reasons
briefly explained in the introduction, we regard (\ref{eq4.3}) as a
complementarity problem (\ref{eq1.4}), (\ref{eq1.5}). In what
follows, we will show that its unique solution is of the form
$(u,\mu)$, where $u$ is defined by (\ref{eq2.09}) and $\mu$ is the
measure corresponding to $K^s$ in the sense of Lemma \ref{lem3.3}.
Using the convexity of $\psi$, we also show  that $\mu$ is absolutely
continuous with respect to the Lebesgue measure, and we give a
formula for its density. To do this, we carefully examine the
process $K^s$.

Assume that $\psi:\BR^d\rightarrow\BR$ is convex. Let $m$ denote
the Lebesgue measure on $\BR^d$, $\nabla_i\psi$ denote the usual
partial derivative with respect to $x_i$, $i=1\dots,d$, and let
$E$ be the set of all $x\in\BR^d$ for which the gradient
\[
\nabla\psi(x)=(\nabla_1\psi(x),\dots,\nabla_d\psi(x))
\]
exists. As $\psi$ is locally Lipschitz continuous, $m(E^c)=0$
and $\nabla\psi=(\psi_{x_1},\dots,\psi_{x_n})$ a.e. (recall that
$\psi_{x_i}$ stands for the partial derivative in the distribution
sense). Moreover, for a.e. $x\in E$, there exists a $d$-dimensional
symmetric matrix $\{H(x)=\{H_{ij}(x)\}$ such that
\begin{equation}
\label{eq3.26} \lim_{E\ni y\rightarrow x}
\frac{\nabla\psi(y)-\nabla\psi(x) -H(x)(y-x)}{|y-x|}=0,
\end{equation}
i.e., $H_{ij}(x)$ are defined as limits through the set where
$\nabla_i\psi$ exists (see, e.g., Alberti and Ambrosio (1999, section 7.9)). By Alexandrov's theorem (see, e.g.,  Alberti and Ambrosio (1999, theorem 7.10)), if $x\in E$ is a point where (\ref{eq3.26}) holds, then $\psi$ has
a second-order differential at $x$ and $H(x)$ is the Hessian matrix
of $\psi$ at $x$, i.e., $H(x)=\{\nabla^2_{ij}\psi(x)\}$.

The second-order derivative of $\psi$ in the distribution sense
$D^2\psi=\{\partial_{x_ix_j}\psi\}_{i,j=1,\dots,d}$ is a matrix of
real-valued Radon measures $\{\mu_{ij}\}_{i,j=1,\dots,d}$ on
$\BR^d$ such that $\mu_{ij}=\mu_{ji}$, and for each Borel set $B$,
$\{\mu_{ij}(B)\}$ is a nonnegative definite matrix (see, e.g.,
Evans and Gariepy (1992, section 6.3)).

Let $(s,x)\in[0,T)\times D$. As $\psi$ is convex, by
Bouleau (1984, theorem 3), there exists a c\`adl\`ag adapted increasing
process $V^s$ on $[s,T]$ such that for every $x\in\BR^d$,
\begin{equation}
\label{eq3.20} \psi(X_t)=\psi(X_s)+\sum_{i=1}^d\int^t_s
\nabla_i\psi(X_{\theta-})\,dX^i_{\theta} +V^s_t,\quad t\in[s,T],
\quad P_{s,x}\mbox{-a.s.}
\end{equation}
Moreover, the process $A^s$ defined as
\[
A^s_t=V^s_t-J^s_t,\quad t\in[s,T],
\]
where
\[
J^s_t=\sum_{s<\theta\le t} \{\psi(X_{\theta})-\psi(X_{\theta-})
-\sum^d_{i=1} \nabla_i\psi(X_{\theta-})\Delta X^i_{\theta}\},
\]
is a continuous increasing process.

\begin{lemma}
\label{lem3.2} Assume that $\psi:\BR^d\rightarrow\BR$ is convex.
Let $V^s,A^s$ be defined as above and let
\[
A^{s,a}_t=\frac12\int^t_s\sum^d_{i,j=1}
a_{ij}X^i_{\theta}X^j_{\theta}\nabla^2_{ij}\psi(X_{\theta})\,d\theta,
\quad t\in[s,T],
\]
Then, $V^s_t-A^{s,a}_t$, $t\ge s$, is an increasing process under
the measure $P_{s,x}$.
\end{lemma}
\begin{proof}
Let $\{\rho_{\varepsilon}\}_{\varepsilon>0}$ be some family of
mollifiers and let
\[
\psi_{\varepsilon}=\psi*\rho_{\varepsilon}, \qquad
\gamma^{\varepsilon}(dx) =\frac12\sum^d_{i,j=1}a_{ij}x_ix_j\,
\nabla^2_{ij}\psi_{\varepsilon}(x)\,dx.
\]
By $J^{s,\varepsilon}$, denote the process defined as $J^s$ but
with $\psi$ replaced by $\psi_{\varepsilon}$, and set
\[
A^{s,\varepsilon}_t=\frac12\sum^d_{i,j}\int^t_sa_{ij}X^i_{\theta}
X^j_{\theta}\nabla^2_{ij}\psi_{\varepsilon}(X_{\theta})\,d\theta,
\quad t\ge s.
\]
By It\^o's formula,
\[
\psi_{\varepsilon}(X_t)=\psi_{\varepsilon}(X_s)+\sum_{i=1}^d\int^t_s
\nabla_i\psi_{\varepsilon}(X_{\theta-})\,dX^i_{\theta}
+V^{s,\varepsilon}_t,\quad t\ge s,
\]
where
\[
V^{s,\varepsilon}_{t}=\frac12\int^{t}_sa_{ij}X^i_{\theta}X^j_{\theta}
\nabla^2_{ij}\psi_{\varepsilon}(X_{\theta})\,d\theta+J^{s,\varepsilon}_{t}
=A^{s,\varepsilon}_t+J^{s,\varepsilon}_t.
\]
Let $T>s$. By Carlen and Protter (1992, theorem 2) and the remarks preceding it, there exist stopping times $\tau_R$ increasing to infinity
$P_{s,x}$-a.s. as $R\uparrow\infty$ such that
$|X_{\cdot\wedge\tau_R-}|\le R$, $V^s_{\cdot\wedge\tau_R-}\le R$
and
\begin{equation}
\label{eq5.4} E_{s,x}\sup_{s\le t\le
T}|\sum_{i=1}^d\int^{t\wedge\tau_R-}_s
\nabla_i(\psi_{\varepsilon}-\psi)(X_{\theta})
(r-\delta_i)X^i_{\theta}\,d\theta
+V^{s,\varepsilon}_{t\wedge\tau_R-}-V^s_{t\wedge\tau_R-}|\rightarrow0
\end{equation}
as $\varepsilon\downarrow0$. Clearly,
\begin{align}
\label{eq5.6} &E_{s,x}\sup_{s\le t\le T}
|\sum_{i=1}^d\int^{t\wedge\tau_R-}_s
\nabla_i(\psi_{\varepsilon}-\psi)
(X_{\theta})(r-\delta_i)X^i_{\theta}\,d\theta| \\
&\qquad\le\sum^d_{i=1}\int^T_s\!\!\int_{B_R}
|\nabla_i(\psi_{\varepsilon}-\psi)(y)(r-\delta_i)y_i|p(t-s,x,y)\,dy.
\nonumber
\end{align}
By Carlen and Protter (1992, lemma), $\sup_{\varepsilon>0}\sup_{|y|\le R}
|\nabla\psi_{\varepsilon}(y)|<\infty$, while  from the proof of
Rockafellar (1970, theorem 25.7) and the fact that $\psi$ is a.e.
differentiable it follows that
$\nabla_i\psi_{\varepsilon}(y)\rightarrow\nabla_i\psi(y)$ for a.e.
$y\in\BR^d$ (see  Grinberg (2013, theorem 4)). Therefore, applying the
Lebesgue dominated convergence theorem shows that the right-hand
side of (\ref{eq5.6}) converges to zero as
$\varepsilon\downarrow0$.  By this and (\ref{eq5.4}),
\begin{equation}
\label{eq5.5} E_{s,x}\sup_{s\le t\le T}
|V^{s,\varepsilon}_{t\wedge\tau_R-}-V^s_{t\wedge\tau_R-}|\rightarrow0.
\end{equation}
As $\psi_{\varepsilon}$ is convex, $J^{s,\varepsilon}$ is an
increasing process. Therefore, for all  $R>0$, $\alpha>0$  and
nonnegative $f\in C_c(\BR^d)$,
\begin{align}
\label{eq3.28} L^{\varepsilon}_R(x)&:=E_{s,x}
\int^{T\wedge\tau_R-}_se^{-\alpha(t-s)}
f(X_t)\,dV^{s,\varepsilon}_t\\
&\ge E_{s,x}\int^{T\wedge\tau_R-}_se^{-\alpha(t-s)}
f(X_t)\,dA^{s,\varepsilon}_t=:P^{\varepsilon}_R(x).\nonumber
\end{align}
By (\ref{eq5.5}) and the Lebesgue dominated convergence theorem,
\begin{equation}
\label{eq3.29} \lim_{\varepsilon\downarrow0}L^{\varepsilon}_R(x)
=E_{s,x}\int^{T\wedge\tau_R-}_se^{-\alpha(t-s)} f(X_t)\,dV^s_t.
\end{equation}
Let $\mu_{ij}=\mu^a_{ij}+\mu^s_{ij}$ be the Lebesgue decomposition
of the measure $\mu_{ij}$ into the absolutely continuous and
singular parts, i.e., $\mu^a_{ij}\ll m$ and $\mu^s_{ij}\bot m$. Set
$[D^2\psi]=\{\mu_{ij}\}_{i,j=1,\dots,d}$ and
$[D^2\psi]_a=\{\mu^a_{ij}\}_{i,j=1,\dots,d}$\,,
$[D^2\psi]_s=\{\mu^s_{ij}\}_{i,j=1,\dots,d}$\,, so that
$[D^2\psi]=[D^2\psi]_a+[D^2\psi]_s$. From the fact that the
matrix-valued measure $[D^2\psi]_s$ is concentrated on some set
$S$ such that $m(S)=0$ and $[D^2\psi]$ is nonnegative definite, it
follows that $[D^2\psi]_s$ is nonnegative definite. Hence, the
matrix-valued measure $[D^2\psi]_s*\rho_{\varepsilon}
=\{\mu^s_{ij}*\rho_{\varepsilon}\}_{i,j=1,\dots,d}$ is nonnegative
definite because, for every Borel set $B\subset\BR^d$ and every
$z=(z_1,\dots,z_d)\in\BR^d$, we have
\begin{align*}
\sum^d_{i,j=1}\mu^s_{ij}*\rho_{\varepsilon}(B)z_iz_j
&=\sum^d_{ij=1}\Big(\int_{\BR^d}\mu^s_{ij}((B-x)
\rho_{\varepsilon}(x)\,dx\Big)z_iz_j\\
&=\int_{\BR^d}\Big(\sum^d_{ij=1}\mu^s_{ij}((B-x)\cap S)
z_iz_j\Big)\rho_{\varepsilon}(x)\,dx\ge0,
\end{align*}
the last inequality being a consequence of the fact that
$[D^2\psi]_s$ is nonnegative definite. Because
$[D^2\psi_{\varepsilon}]=[D^2\psi]*\rho_{\varepsilon} =
[D^2\psi]_a*\rho_{\varepsilon}+[D^2\psi]_s*\rho_{\varepsilon}$ and
by Theorem 1 in Section 6.4 in Evans and Gariepy (1992), the density of
$[D^2\psi]_a$ is given by the Hessian matrix
$H=\nabla^2_{ij}\psi$, it follows that
\begin{equation}
\label{eq3.30} \gamma^{\varepsilon}(dx) \ge \frac12\sum^d_{i,j=1}
a_{ij}x_ix_j((\nabla^2_{ij}\psi))*\rho_{\varepsilon}(x)\,dx.
\end{equation}
For $t\ge s$, set
\[
B^{s,\varepsilon}_t=\frac12\sum^d_{i,j=1}\int^t_sa_{ij}X^i_{\theta}
X^j_{\theta}((\nabla^2_{ij}\psi)*\rho_{\varepsilon})(X_{\theta})\,d\theta
\]
and
\[
A^{s,\varepsilon,k}_t=\int^t_s\fch_{\{|X_{\theta}|\le
k\}}\,dA^{s,\varepsilon}_{\theta}, \qquad
B^{s,\varepsilon,k}_t=\int^t_s\fch_{\{|X_{\theta}|\le
k\}}\,dB^{s,\varepsilon}_{\theta}.
\]
If $\alpha>0$, then
$E_{s,x}\int^{\infty}_se^{-\alpha(t-s)}\,d(A^{s,\varepsilon,k}_t
+B^{s,\varepsilon,k}_t)<\infty$, so from (\ref{eq3.30}) it follows
that, for every  $k>0$ and every nonnegative $f\in C_c(\BR^d)$,
\[
E_{s,x}\int^{\infty}_se^{-\alpha(t-s)}
f(X_t)\,dA^{s,\varepsilon,k}_t\ge
E_{s,x}\int^{\infty}_se^{-\alpha(t-s)}
f(X_t)\,dB^{s,\varepsilon,k}_t.
\]
Therefore, from the proof of Revuz and Yor (1991, proposition X.1.7), it
follows that $A^{s,\varepsilon,k}-B^{s,\varepsilon,k}$ is
increasing for $k>0$. Letting $k\rightarrow\infty$ shows that the
process $A^{s,\varepsilon}-B^{s,\varepsilon}$ is increasing.
Consequently,
\[
P^{\varepsilon}_R(x)\ge
E_{s,x}\int^{T\wedge\tau_R}_se^{-\alpha(t-s)} f(X_t)\,
dB^{s,\varepsilon}_t.
\]
Let $g_{ij}(y)=\liminf_{\varepsilon\downarrow0}
(\nabla^2_{ij}\psi)*\rho_{\varepsilon}(y)$, $y\in\BR^d$. Applying
Fatou's lemma and the monotone convergence theorem, we conclude
from the above inequality that
\[
\liminf_{\varepsilon\downarrow0}P^{\varepsilon}_R(x)\ge
E_{s,x}\int^{T\wedge\tau_R}_se^{-\alpha(t-s)} f(X_t)
\frac12\sum^d_{ij=1}a_{ij}X^i_tX^j_tg_{ij}(X_t)\,dt.
\]
Because
$(\nabla^2_{ij}\psi)*\rho_{\varepsilon}\rightarrow\nabla^2_{ij}\psi$,
$m$-a.e. as $\varepsilon\downarrow0$, $g_{ij}=\nabla^2_{ij}\psi$,
$m$-a.e. From this and the fact that $p(t,x,\cdot)\ll m$, it
follows that, on the right-hand side of the above inequality, we may
replace $g_{ij}$ by $\nabla^2_{ij}\psi$.
Thus,
\[
\liminf_{\varepsilon\downarrow0}P^{\varepsilon}_R(x)\ge
E_{s,x}\int^{T}_se^{-\alpha(t-s)}f(X_t)\,dA^{s,a}_t.
\]
Combining this with (\ref{eq3.28}) and (\ref{eq3.29}), we get
\[
E_{s,x}\int^{T\wedge\tau_{R-}}_se^{-\alpha(t-s)} f(X_t)\,dV^s_t\ge
E_{s,x}\int^{T\wedge\tau_{R-}}_se^{-\alpha(t-s)}
f(X_t)\,dA^{s,a}_t.
\]
Letting $T\rightarrow\infty$ in the above inequality yields
\[
E_{s,x}\int^{\infty}_se^{-\alpha(t-s)}
f(X_t)\,dV^s_{t\wedge\tau_{R-}}\ge
E_{s,x}\int^{\infty}_se^{-\alpha(t-s)}
f(X_t)\,dA^{s,a}_{t\wedge\tau_{R-}}
\]
for any $x\in D$ and nonnegative $f\in C_c(\BR^d)$. From this and
arguments from the proof of  Revuz and Yor (1991, proposition X.1.7), it
follows that, for every $R>0$, the process
$V^{s,a}_{t\wedge\tau_{R-}}-A^{s,a}_{t\wedge\tau_{R-}}$ is
increasing, and hence that $V^s-A^{s,a}$ is increasing.
\end{proof}

\begin{remark}
If $\nu=0$, then $A^s_t=A^{s,a}_t$, $t\in[s,T]$ (see
Klimsiak and Rozkosz (2016)).
\end{remark}

\begin{definition} (a) We say that a pair $(u,\mu)$, where
$u\in\WW^{0,1}_{\varrho}\cap C([0,T]\times\BR^d)$ and $\mu$ is a
Radon measure on $Q_{T}$, is a variational solution of  problem
(\ref{eq4.3}) if
\[
u(T)=\psi,\quad u\ge\psi,\quad
\int_{Q_T}(u-\psi)\varrho^2\,d\mu=0
\]
and the equation
\[
\partial_su+L u=ru-\mu
\]
is satisfied in the weak sense, i.e., for every $\eta\in
C_{c}^{\infty}(Q_{T})$,
\[
\int^T_0\langle\partial_tu(t),\eta(t)\rangle\,dt
+\int^T_0B_{\varrho}(u(t),\eta(t))\,dt=r\int_{Q_T}
u\eta\varrho^2\,dt\,dx-\int_{Q_{T}}\eta\varrho^{2}\,d\mu,
\]
where $\langle\cdot,\cdot\rangle$ denotes the duality pairing
between $H^{-1}_{\varrho}$ and $H^1_{\varrho}$.
\smallskip\\
(b) If $\mu$ in the above definition admits a density (with
respect to the Lebesgue measure) of the form
$\Psi_u(t,x)=\Psi(t,x,u(t,x))$ for some measurable
$\Phi:Q_T\times\BR\rightarrow\BR_+$, then we say that $u$ is a
variational solution, in the space $\WW^{0,1}_{\varrho}$, to the
semilinear problem
\begin{equation}
\label{eq3.14}
\partial_su+L u=ru-\Phi_u,\quad u(T,\cdot)=\psi,\quad
u\ge\psi.
\end{equation}
\end{definition}

Let
\[
\LL_{BS}=\frac12\sum^d_{i,j=1}a_{ij}x_ix_j\nabla^2_{ij}
+\sum^d_{i=1}(r-\delta_i)x_i\nabla_i
\]
and for $\psi:\BR^d\rightarrow\BR$ such that $\psi_{|D_{\iota}}$
is convex for every $\iota\in I$ set
\begin{equation}
\label{eq3.7} \Psi=-r\psi+\LL_{BS}\psi.
\end{equation}
In the sequel, $\Psi^{-}$ stands for $(-\Psi)\vee0$.

\begin{theorem}
\label{th5.3} Let  $\beta>p$ and let $(a,\nu,\gamma)$ satisfy
\mbox{\rm(\ref{eq2.2}), (\ref{eq2.02}), (\ref{eq2.06}),
(\ref{eq4.33})}. Assume that  $\psi$ is a measurable function
satisfying \mbox{\rm(\ref{eq3.1})} and $\tilde\psi\in\tilde
H^1_{\rho}$, where $\tilde\psi$ is defined by
\mbox{\rm(\ref{eq4.22})}. Moreover, assume that, for every
$\iota\in I$, the restriction of $\psi$ to $D_{\iota}$ is a convex
function, which can be extended to a finite convex function on all
of $\BR^d$, and that
\begin{equation}
\label{eq5.7} \Psi^{-}\in L^2_{\varrho}.
\end{equation}
Then, $u$ defined by \mbox{\rm(\ref{eq2.09})} has the following
properties.
\begin{enumerate}
\item[\rm(i)]$u\in W^{1,2}_{\varrho}$ and $u$ is a unique
variational solution of the problem
\begin{equation}
\label{eq3.6}
\partial_tu+L u=ru+\fch_{\{u=\psi\}}\fch_{\{\Psi<0\}}(-\Psi^{-}+L_Iu),\quad
u(T)=\psi,\quad u\ge\psi,
\end{equation}
where $\Psi$ is defined by \mbox{\rm(\ref{eq3.7})}.
\item[\rm(ii)]Let $Y_t=u(t,X_t)$, $t\in[0,T]$, and let $M$
be a c\`adl\`ag martingale defined as
\begin{align*}
M_t&=E_{s,x}\Big(\psi(X_T) +\int^T_s\{-ru(\theta,X_{\theta})\\
&\quad+\fch_{\{u(\theta,X_{\theta})=\psi(X_{\theta})\}}
\fch_{\{\Psi(X_{\theta})<0\}}
(\Psi^{-}-L_Iu)(\theta,X_{\theta})\}\,d\theta\big|\FF^s_t\Big)
-u(s,X_s)
\end{align*}
for $t\in[s,T]$. Then, for every $(s,x)\in[0,T)\times P$, the pair
$(Y,M)$ is a solution of the BSDE
\begin{align*}
Y_t&=\psi(X_T)-\int^T_trY_{\theta}\,d\theta
-\int^T_t\fch_{\{u(\theta,X_{\theta})=\psi(X_{\theta})\}}
\fch_{\{\Psi(X_{\theta})<0\}}
(\Psi^{-}-L_Iu)(\theta,X_{\theta})\,d\theta\nonumber\\
&\quad-\int^T_tdM_{\theta},\quad t\in[s,T],\quad
P_{s,x}\mbox{\rm-a.s.}
\end{align*}
\end{enumerate}
\end{theorem}
\begin{proof}
Let $(s,x)\in[0,T)\times D$, $(Y=u(\cdot,X),M^{s},K^{s})$ be a
solution of (\ref{eq3.06}) (see Theorem \ref{th3.1}), and let
$V^s$, $A^{s,a}$ be the processes defined in Lemma \ref{lem3.2}.
Because
\[
Y_t=Y_s+\int^t_srY_{\theta}\,d\theta-\int^t_sdK^{s}_{\theta}
+\int^t_sdM^{s}_{\theta},\quad t\in[s,T],\quad
P_{s,x}\mbox{-a.s.},
\]
it follows from (\ref{eq2.03}) and (\ref{eq3.20}) that
\begin{align}
\label{eq3.2} Y_t-\psi(X_t)&=Y_s-\psi(X_s)-\int^t_s
\{-rY_{\theta}+\LL_{BS}\psi(X_{\theta})\}\,d\theta\\
&\quad-\int^t_sdK^{s}_{\theta}
+\int^t_sd(M^{s}_{\theta}-M^{\psi}_{\theta})
-\int^t_s\,d(V^s_{\theta}-A^{s,a}_{\theta}),\nonumber
\end{align}
where
\[
M^{\psi}_t=\sum^d_{i=1}\int^t_s
\nabla_i\psi(X_{\theta-})\,d(M^{c,i}_{\theta}
+M^{d,i}_{\theta}).
\]
By the Meyer-Tanaka formula (see Protter (2004, theorem IV.68)) and the
fact that $Y_t\ge\psi(X_t)$ for $t\in[s,T]$, we have
\begin{align*}
(Y_t-\psi(X_t))^+-(Y_s-\psi(X_s))^{+}
&=\int^t_s\fch_{\{Y_{\theta-}>\psi(X_{\theta-})\}}\,
d(Y_{\theta}-\psi(X_{\theta}))+B^s_t\nonumber\\
& =(Y_t-\psi(X_t))^{+}-(Y_s-\psi(X_s))^{+}\nonumber\\
&\quad-\int^t_s\fch_{\{Y_{\theta-}=\psi(X_{\theta-})\}}\,
d(Y_{\theta}-\psi(X_{\theta}))+B^s_t,
\end{align*}
where $B^s$ is some adapted c\`adl\`ag increasing process on
$[s,T]$ such that $B^s_s=0$. From the above, it follows that, in
fact,
\[
B^s_t=\int^t_s\fch_{\{Y_{\theta-}=\psi(X_{\theta-})\}}\,
d(Y_{\theta}-\psi(X_{\theta})).
\]
By the above equality and (\ref{eq3.2}),
\begin{align*}
B^s_t&=-\int^t_s\fch_{\{Y_{\theta-}=\psi(X_{\theta-})\}}
(-rY_{\theta}+\LL_{BS}\psi(X_{\theta}))\,d\theta\\
&\quad-\int^t_s\fch_{\{Y_{\theta-}
=\psi(X_{\theta-})\}}\,d(K^{s}_{\theta}+V^s_{\theta}-A^{s,a}_{\theta})
+\int^t_s\fch_{\{Y_{\theta-}=\psi(X_{\theta-})\}}\,
d(M^{s}_{\theta}-M^{\psi}_{\theta}).
\end{align*}
Because
\[
\int^t_s\fch_{\{Y_{\theta-}=\psi(X_{\theta-})\}}\,dK^{s}_{\theta}
=K^{s}_t,
\]
it follows that
\begin{align*}
&K^{s}_t+B^s_t+\int^t_s \fch_{\{Y_{\theta-}
=\psi(X_{\theta-})\}}\,d(V^s_{\theta}-A^{s,a}_{\theta})\\
&\qquad=-\int^t_s\fch_{\{Y_{\theta-}=\psi(X_{\theta-})\}}
(-rY_{\theta}+\LL_{BS}\psi(X_{\theta}))\,d\theta
+\int^t_s\fch_{\{Y_{\theta-}=\psi(X_{\theta-})\}}\,
d(M^{s}_{\theta}-M^{\psi}_{\theta}).
\end{align*}
Write
\[
C^s_t=B^s_t+\int^t_s\fch_{\{Y_{\theta-}
=\psi(X_{\theta-})\}}\,d(V^s_{\theta}-A^{s,a}_{\theta}),\quad
t\in[s,T],
\]
and by $\tilde C^s$ denote the compensator of $C^s$. Then,
\begin{align}
\label{eq3.21} K^{s}_t+\tilde C^s_t
&=-\int^t_s\fch_{\{Y_{\theta-}=\psi(X_{\theta-})\}}
(-rY_{\theta}+\LL_{BS}\psi(X_{\theta}))\,d\theta\\
&\quad+\int^t_s\fch_{\{Y_{\theta-}=\psi(X_{\theta-})\}}\,
d(M^{s}_{\theta}-M^{\psi}_{\theta}) +\tilde C^s_t-C^s_t.
\nonumber
\end{align}
As $K^{s}$ is a continuous increasing process, the left-hand
side of equality (\ref{eq3.21}) is a special semimartingale under
$P_{s,x}$. Similarly, the right-hand side is a special
semimartingale under $P_{s,x}$. Suppose that $x\in D_{\iota}$ for
some $\iota\in I$ and denote by $\psi_{\iota}$ a convex function
on $\BR^d$ whose restriction to $D_{\iota}$ coincides with
$\psi_{|D_{\iota}}$. By Remark \ref{rem2.1}(i), under the measure
$P_{s,x}$, the processes $V^s,A^{s,a}$ defined in Lemma
\ref{lem3.2}  coincide with the processes defined analogously to
$V^s,A^{a,s}$, but  with $\psi$ replaced by $\psi_{\iota}$.
Therefore, by Lemma \ref{lem3.2} applied to $\psi_{\iota}$, the
process $C^s$ is increasing under $P_{s,x}$, and hence $\tilde
C^s$ is also increasing (see Protter, (2004, p. 120)).
Because the decomposition of the special semimartingale is unique
and the process on the left-hand side of (\ref{eq3.21}) is
increasing, it follows from (\ref{eq3.21}) that
\[
dK^{s}_t+d\tilde C^s_t=\fch_{\{Y_{t-}=\psi(X_{t-})\}}
(-rY_{t}+\LL_{BS}\psi(X_{t}))^{-}\,dt.
\]
From the above equality and the fact that $\tilde C_{s,\cdot}$ is
increasing, we conclude that
\begin{align}
\label{eq3.4} 0\le dK^{s}_t&\le\fch_{\{Y_{t-}=\psi(X_{t-})\}}
(-rY_{t}+\LL_{BS}\psi(X_{t}))^{-}\,dt  \\
&=\fch_{\{u(t,X_t)=\psi(X_t)\}}(-r\psi(X_t)+\LL_{BS}\psi(X_{t}))^{-}\,dt.
\nonumber
\end{align}
Define $\Psi$ by (\ref{eq3.7}) and by $\mu$ denote the measure of
Lemma \ref{lem3.3}. By (\ref{eq3.4}) and Lemma \ref{lem3.3}, for
every $f\in C_c((0,T)\times D)$, we have
\begin{align*}
\int^T_0\!\!\int_{D}f(t,y)p(t,x,y)\,d\mu(t,y)
&=E_{0,x}\int^T_0f(t,X_t)\,dK^0_t\\
&\le E_{0,x}\int^T_0 f(t,X_t)
\fch_{\{u(t,X_t)=\psi(X_t)\}}\Psi^{-}(X_t)\,dt\\
&=\int^T_0\!\!\int_{D} f(t,y)p(t,x,y)\,d\nu(t,y),
\end{align*}
where $\nu=\fch_{\{u(t,y)=\psi(y)\}}\Psi^{-}(y)\,dt\,dy$. Let
$x\in D$. Then, $p(\cdot,x,\cdot)>0$  by Remark \ref{rem2.1}, so
from the above inequality it follows that $\mu(B)\le\nu(B)$ for
every Borel set $B\subset(0,T)\times D$. In particular,
$\mu\ll\nu$, so by the Radon-Nikodym theorem, there exists a
nonnegative measurable $\alpha:(0,T)\times\BR^d\rightarrow\BR$
such that
\begin{equation}
\label{eq5.14} d\mu(t,y)=\alpha(t,y)
\fch_{\{u(t,y)=\psi(y)\}}\Psi^{-}(y)\,dt\,dy.
\end{equation}
In fact, as $\mu\le\nu$, we have $\alpha\le1$, $\nu$-a.e., and
hence $\alpha\le1$ a.e. on the set where
$\fch_{\{u=\psi\}}\Psi^{-}>0$. By (\ref{eq3.25}) and
(\ref{eq5.14}),
\[
E_{s,x}K^s_T=E_{s,x}\int^T_s
\alpha(t,X_t)\fch_{\{u(t,X_t)=\psi(X_t)\}}\Psi^{-}(X_t)\,dt
\]
for $(s,x)\in(0,T)\times D$. Therefore, by (\ref{eq3.06}) and
Theorem \ref{th3.1},
\begin{equation}
\label{eq5.15} u(s,x)=E_{s,x}\Big(\psi(X_T)+\int^T_s\{-ru(t,X_t)
+\alpha(t,X_t)\fch_{\{u(t,X_t)=\psi(X_t)\}}\Psi^{-}(X_t)\}\,dt\Big)
\end{equation}
for every $(s,x)\in(0,T)\times D$. Because (\ref{eq5.7}) is
satisfied, it follows from Proposition \ref{prop4.9} that $u$ is a
unique, in the space $\WW^{0,1}_{\varrho}$, solution of the Cauchy
problem
\begin{equation}
\label{eq3.22}
\partial_su+(L_{BS}+L_I) u-ru=-\alpha\fch_{\{u=\psi\}}\Psi^{-},
\quad u(T)=\psi.
\end{equation}
In fact,  $u\in W^{1,2}_{\varrho}$ because $\tilde\psi\in\tilde
H^1_{\rho}$. We now show that
\begin{equation}
\label{eq3.33}
-\alpha\fch_{\{u=\psi\}}\Psi^{-}=\fch_{\{u=\psi\}}
\fch_{\{\Psi<0\}}(-\Psi^{-}+L_Iu)
\quad\mbox{a.e. on } Q_T.
\end{equation}
As $u\in W^{1,2}_{\varrho}$, $u(t,\cdot)\in H^2_{\varrho}$ for
a.e. $t\in(0,T)$. Therefore, by Remark (ii) following Theorem 4 in
Section 6.1 in Evans and Gariepy (1992), the distributional derivatives
$\partial_{x_i}u$, $\partial^2_{x_ix_j}u$ are a.e. equal to the
approximate derivatives $\nabla^{ap}_iu$, $(\nabla^{ap})^2_{ij}u$.
Let $\LL^{ap}_{BS}$ denote the operator defined as $\LL_{BS}$ but
with $\nabla_i$, $\nabla_{ij}$ replaced by $\nabla^{ap}_i$,
$(\nabla^{ap})^2_{ij}$. Then, $u$ is a variational solution of
(\ref{eq3.22}) with $L_{BS}$ replaced by $\LL^{ap}_{BS}$, i.e., a
solution of the problem
\[
\partial_su+\LL^{ap}_{BS}u+L_Iu-ru
=-\alpha\mathbf{1}_{\{u=\psi\}}\Psi^{-},
\quad u(T)=\psi.
\]
Therefore, using the argument from the proof of Proposition
\ref{prop4.9}(iii), we  show that
\begin{equation}
\label{eq4.56}
\partial_su+\LL^{ap}_{BS}u+L_Iu-ru=-\alpha\fch_{\{u=\psi\}}\Psi^{-}
\quad\mbox{a.e. on } Q_T.
\end{equation}
As $\psi$ is convex, $\psi\in BV_{loc}(\BR^d)$ as a locally
Lipschitz continuous function and, by Theorem 3 in Section 6.3 in
Evans and Gariepy (1992), $\psi_{x_i}\in BV_{loc}(\BR^d)$, $i=1,\dots,d$.
Therefore, $\psi$ is twice approximately differentiable a.e. by
Theorem 4 in Section 6.1 in Evans and Gariepy (1992). Therefore, from Theorem 3 in Section 6.1 in Evans and Gariepy (1992), it follows that $\LL^{ap}u=\LL^{ap}\psi$
a.e. on $\{u=\psi\}$. Furthermore, as $\psi$ is convex,
$\LL^{ap}_{BS}\psi=\LL_{BS}\psi$ a.e. on $\BR^d$ by Remark (i)
following Theorem 4 in Section 6.1 in Evans and Gariepy (1992). Furthermore,
as  the functions  $u$ and $(t,x)\mapsto\psi(x)$, together with
their first distributional derivatives, are integrable on each
relatively compact open set of $(0,T)\times D$, it follows from
Theorem 4(iv) in Section 4.2 in Evans and Gariepy (1992) that
$\partial_tu=\partial_t\psi=0$ a.e. on $\{u=\psi\}$. Therefore, by
(\ref{eq4.56}),
\[
\LL_{BS}\psi+L_Iu-r\psi=-\alpha\Psi^{-}\quad\mbox{a.e. on }
\{u=\psi\}.
\]
This and (\ref{eq3.7}) imply that
$-\alpha\fch_{\{u=\psi\}}\Psi^{-}=\fch_{\{u=\psi\}}(\Psi+L_Iu)$
a.e. on  $Q_T$, from which (\ref{eq3.33}) follows. Combining
(\ref{eq3.22}) with (\ref{eq3.33}), we see that $u$ satisfies
(\ref{eq3.6}), which completes the proof of (i). Because
$p(s,x,t,\cdot)\ll m$, it follows from (\ref{eq5.15}) and
(\ref{eq3.33}) that
\begin{align*}
u(s,x)&=E_{s,x}\Big(\psi(X_T)\\
&\quad+\int^T_s\{-ru(t,X_t) +\fch_{\{u(t,X_t)=\psi(X_t)\}}
\fch_{\{\Psi(X_{t})<0\}}(\Psi^{-}-L_Iu)(t,X_t)\}\,dt\Big).
\end{align*}
From this we deduce (ii) (see the beginning of the proof of
Proposition \ref{prop4.9}).
\end{proof}


\begin{remark}
By (\ref{eq3.33}), $ \Psi^{-}-L_Iu\ge0$ a.e. on
$\{u=\psi\}\cap\{\Psi<0\}$.
\end{remark}

\begin{proposition}
\label{prop5.5} Assume that \mbox{\rm(\ref{eq2.2}),
(\ref{eq2.02}), (\ref{eq3.1}), (\ref{eq2.06})}  are  satisfied and
$\psi$ is a continuous function such that $\psi(x)>0$ for some
$x\in D_{\iota}$. Then, for every $s\in[0,T)$, $\{x\in
D_{\iota}:u(s,x)=\psi(x)\}\subset\{x\in D_{\iota}:\psi(x)>0\}$.
\end{proposition}
\begin{proof}
We use the argument from the proof of Villeneuve (1999, proposition 1.1). Fix
$s\in[0,T)$, $x\in D_{\iota}$ and set $O=\{y\in
D_{\iota}:\psi(y)>0\}$. Then, $O$ is a nonempty open set. Because,  by
Remark \ref{rem2.1}, the density of the distribution of $X_T$ under
$P_{s,x}$  is strictly positive on $D_{\iota}$, $P_{s,x}(X_{T}\in
O)>0$. By this and (\ref{eq2.09}), $u(s,x)\ge
e^{-r(T-s)}E_{s,x}\psi(X_T)>0$, from which the proposition
follows.
\end{proof}
\medskip

From Proposition \ref{prop5.5}, it follows that, for nontrivial
$\psi$, we may replace $\fch_{\{u=\psi\}}$ by
$\fch_{\{u=\psi\}\cap\{\psi>0\}}$ in the formulation of Theorem
\ref{th5.3}. This often makes the computation of $\Psi$ easier
because usually $\psi$ is regular on $\{\psi>0\}$ (see examples in
Section \ref{sec6}). Moreover, because in the model (\ref{eq1.1}),
the initial prices are strictly positive, it follows from Remark
\ref{rem2.1}(i) that it suffices to compute $\Psi(x)$ for
$x=(x_1,\dots,x_d)$ such that $x_i>0$, $i=0,\dots,d$.

\section{The early exercise premium formula}
\label{sec6}

Let $\eta$ denote the payoff process for  an American option with
payoff function $\psi$, i.e.,
\[
\eta_t=e^{-r(t-s)}\psi(X_t),\quad t\in[s,T].
\]
Let $V$ be the process defined by (\ref{eq2.08})  and let
$Y_t=u(t,X_t)$, $t\in[s,T]$, where $u$ is defined by
(\ref{eq2.09}). By (\ref{eq2.10}) and (\ref{eq2.11}), the process
\[
\tilde V_t=e^{-r(t-s)}V_t,\quad t\in[s,T]
\]
is an $(\FF^s_t)$-supermartingale under $P_{s,x}$. In fact, by
Klimsiak (2015, lemma 2.8), it is the smallest supermartingale on
$[s,T]$ that dominates the process $\eta$, i.e., $\tilde V$ is the
Snell envelope for $\eta$. Applying Theorem \ref{th5.3}(ii) gives
the following representation for $\tilde V$.

\begin{corollary}
Let the assumptions of  Theorem \ref{th5.3} hold. Then, for all
$s\in[0,T)$ and $x\in D$,
\begin{align}
\label{eq6.1} \tilde
V_t&=E_{s,x}\Big(e^{-r(T-s)}\psi(X_T) \\
&\qquad+\int^T_te^{-r(\theta-s)}
\fch_{\{u(\theta,X_{\theta})=\psi(X_{\theta})\}}\fch_{\{\Psi(X_{\theta})<0\}}
(\Psi^{-}(X_{\theta})-L_Iu(\theta,X_{\theta}))\,d\theta\big|\FF^s_t\Big).
\nonumber
\end{align}
\end{corollary}

Putting $t=s$ in (\ref{eq6.1}), we get the following early exercise
premium representation formula.

\begin{corollary}
Let the assumptions of  Theorem \ref{th5.3} hold. Then, for all
$s\in[0,T)$ and $x\in D$,
\begin{align}
\label{eq6.2} u(s,x)&=u^E(s,x)\\
&\quad+E_{s,x}\int^T_se^{-r(t-s)} \fch_{\{u(t,X_t)=\psi(X_t)\}}
\fch_{\{\Psi(X_{t})<0\}}(\Psi^{-}(X_t)-L_Iu(t,X_t))\,dt,\nonumber
\end{align}
where
\[
u^E(s,x)=E_{s,x}e^{-r(T-s)}\psi(X_T)
\]
is the value of the European option with payoff function $\psi$
and expiration time $T$.
\end{corollary}

We close this section with some examples of continuous payoff
functions satisfying (\ref{eq3.1}). Using the results of Sections
4 and 5 in Rockafellar (1970), one can easily check that, apart from put
index option, in all examples $\psi$ is convex. In the case of put
index option, $\psi_{|D_{\iota}}$ are convex and can be extended
to convex functions on all of $\BR^d$. In each example, we have
computed the corresponding function $\Psi^{-}$ on the set
$\{\psi>0\}$ (see Proposition \ref{prop5.5}). From formulas for
$\Psi^{-}$, it will be clear that, in each case, $\Psi^{-}$ satisfies
(\ref{eq3.1}), and hence by Remark \ref{rem4.2}, $\Psi^{-}$
satisfies (\ref{eq5.7}). Moreover, the payoff functions in
Examples \ref{ex6.3} and \ref{ex6.4} are Lipschitz continuous, so
by Remark \ref{rem4.8}, $\tilde \psi\in\tilde H^1_{\rho}$ if
$\beta>p$. In Example \ref{ex6.5}, the payoff function $\psi$ is
not Lipschitz continuous, but the fact that $\tilde \psi\in\tilde
H^1_{\rho}$ follows directly from the formula for $\psi$.
Summarizing,  the payoff functions given below satisfy the
assumptions of Theorem \ref{th5.3}, so in each case, the results of
Sections \ref{sec4} and \ref{sec6} apply to exponential L\'evy
models satisfying (\ref{eq2.2}), (\ref{eq2.02}) and
(\ref{eq2.06}), (\ref{eq4.33}) with some $\varepsilon>0$,
$\beta>p$.

\begin{example}
\label{ex6.3} In the following examples $p=0$, so we can take
arbitrary $\beta>0$ in condition (\ref{eq4.33}). Therefore,
(\ref{eq6.1}), (\ref{eq6.2}) hold true  if (\ref{eq4.33}) is
satisfied  for some $\beta>0$ and (\ref{eq2.06}) is satisfied for
some $\varepsilon>0$. Also note that, in the case that $d=1$, the examples
below reduce to the American put considered in Lamberton and Mikou (2013).
\begin{enumerate}
\item
\underline{Min options (put)}
\[
\psi(x)=(K-\min\{x_1,\dots,x_d\})^{+}, \]
\[
\Psi^{-}(x)=\big(r K
-\sum_{i=1}^{d}\delta_{i}\mathbf{1}_{C_{i}}(x)x_i)^{+},
\quad\mbox{where}\quad C_{i}=\{x\in\BR^d: x_{i}<x_{j},\, j\neq
i\},
\]
if $x_i\ge0$ for $i=1,\dots,d$, and $\psi(x)=K$, otherwise.

\item \underline{Index options (put)} \smallskip\\
For simplicity, we consider the case $d=2$. Below,
$w_1\ge0,w_2\ge0$.
\[
\psi(x)=\big(K-\sum_{i=1}^{2} w_{i}x_i\big)^{+},\quad
\Psi^-(x)=\big(r K-\sum_{i=1}^{2} w_{i}\delta_{i}x_i\big)^{+}
\]
if $x_1\ge0,x_2\ge0$, $\psi(x)=(K-w_1x_1)^+$ if $x_1\ge0,x_2<0$,
$\psi(x)=(K-w_2x_2)^+$ if $x_2\ge0,x_1<0$ and $\psi(x)=K$ if
$x_1<0,x_2<0$.
\end{enumerate}
\end{example}

\begin{example}
\label{ex6.4} In the following examples $p=1$, so $\beta>1$.
Therefore, (\ref{eq6.1}), (\ref{eq6.2}) hold true  if
(\ref{eq4.33}) is satisfied for some $\beta>1$.
\begin{enumerate}

\item \underline{Spread option (put)}
\[
\psi(x)=\big(K-\sum_{i=1}^{2} w_{i}x_i\big)^{+},\quad
\Psi^-(x)=\big(r K-\sum_{i=1}^{d} w_{i}\delta_{i}x_i\big)^{+}
\]
(Here, $w_i\in\BR$, $i=1,\dots,d$).

\item \underline{Index options and spread options (call)}
\[
\psi(x)=\big(\sum_{i=1}^{d} w_{i}x_i-K\big)^{+},\quad
\Psi^-(x)=\big(\sum_{i=1}^{d} w_{i}\delta_{i}x_i-r K\big)^{+}
\]
(Here, $w_i\in\BR$, $i=1,\dots,d$).

\item \underline{Max options (call)}
\[
\psi(x)=(\max\{x_1,\dots,x_d\}-K)^{+}
\]
\[
\Psi^{-}(x)=\big(\sum_{i=1}^{d}
\delta_{i}\mathbf{1}_{B_{i}}(x)x_i-r K\big)^{+},
\]
where $B_{i}=\{x\in\BR^d: x_{i}>x_{j},\, j\neq i\}$.

\item \underline{Multiple strike options}
\[
\psi(x)=(\max\{x_1-K_{1},\dots, x_d-K_{d}\})^{+},
\]
\[
\Psi^{-}(x)=\big(\sum_{i=1}^{d}
\mathbf{1}_{B_{i}}(x-K)(\delta_{i}x_i-r K_{i})\big)^{+},
\]
where $K=(K_{1},\dots,K_{d})$ and the sets $B_i$ are defined as in
the preceding example.
\end{enumerate}
\end{example}


\begin{example}
\label{ex6.5} \underline{Power-product options}
\[
\psi(x)=(|x_1\cdot\ldots\cdot x_d|^{\gamma}-K)^+\quad \mbox{for
some }\gamma>1.
\]
If $x\in D_{\iota}$ with $\iota=(i_1,\dots,i_d)\in\{0,1\}^d$ then
\[
\Psi^{-}(x)=\big((r-\gamma\sum_{i=1}^d(r-\delta_i-a_{ii})
-\gamma^2\sum^d_{i,j=1}a_{ij})f(x)-r K\big)^{+},
\]
where $f(x)=((-1)^{|\iota|}\,x_1\cdot\ldots\cdot x_d)^{\gamma}$
and $|\iota|=i_1+\ldots+ i_d$. From the formula for $\psi$ and the
fact that the geometric mean is less then or equal to the
arithmetic mean, it follows that $\psi$ satisfies (\ref{eq3.1})
with $p=\gamma d$. It is also clear that, if $\beta>p$, then
$\tilde\psi$ defined by (\ref{eq4.22}) belongs to the space
$\tilde H^1_{\rho}$. Therefore, (\ref{eq6.1}), (\ref{eq6.2}) hold
true  if (\ref{eq4.33}) is satisfied for some $\beta>\gamma d$.
\end{example}

\end{document}